\documentclass[reqno,a4paper, 11pt]{amsart}

\usepackage[a4paper=true,pdfpagelabels]{hyperref}
\usepackage{graphicx}

\usepackage[ansinew]{inputenc}
\usepackage{amsfonts,epsfig}
\usepackage{latexsym}
\usepackage{amsmath}
\usepackage{amssymb}

\newcommand{\C}{\mathbb{C}}

\newcommand{\D}{\mathbb{D}}
\newcommand{\DD}{\widehat{\mathcal{D}}}
\newcommand{\EE}{\mathcal{E}}
\newcommand{\T}{\mathbb{T}}
\renewcommand{\H}{\mathcal{H}}
\newcommand{\CC}{\mathcal{C}}

\newcommand{\Z}{\mathbb{Z}}

\newcommand{\vp}{\vp}

\newcommand{\NN}{\mathcal{N}}

\newcommand{\N}{\mathbb{N}}

\def\a{\alpha}       \def\b{\beta}        \def\g{\gamma}
\def\d{\delta}           \def\e{\varepsilon} 
     \def\om{\omega}      
       \def\t{\theta}       
         \def\r{\rho}         \def\z{\zeta}
                  \def\vp{\varphi}

\def\muov{\mu^{\om}_v}

\def\C{{\mathbb C}}

\newtheorem{theorem}{Theorem}
\newtheorem{lemma}[theorem]{Lemma}

\newtheorem{lettertheorem}{Theorem}
\newtheorem{letterlemma}[lettertheorem]{Lemma}

\numberwithin{equation}{section}

\begin{document}

\title[Embedding Bergman spaces into tent spaces]{Embedding Bergman spaces into tent spaces}


\thanks{This research was supported in part by the Ram\'on y Cajal program
of MICINN (Spain); by Ministerio de Edu\-ca\-ci\'on y Ciencia, Spain, projects
MTM2011-25502 and MTM2011-26538;  by   La Junta de Andaluc{\'i}a, (FQM210) and
(P09-FQM-4468);  by Academy of Finland project no. 268009,  by V\"ais\"al\"a Foundation of Finnish Academy of Science and Letters, and by Faculty of Science and Forestry of University of Eastern Finland project no. 930349.
}

\date{\today}


\author{Jos\'e \'Angel Pel\'aez}
\address{Departamento de An\'alisis Matem\'atico, Universidad de M\'alaga, Campus de
Teatinos, 29071 M\'alaga, Spain} \email{japelaez@uma.es}

\author{Jouni R\"atty\"a}
\address{University of Eastern Finland, P.O.~Box 111, 80101 Joensuu, Finland}
\email{jouni.rattya@uef.fi}

\author{Kian Sierra}
\address{University of Eastern Finland, P.O.~Box 111, 80101 Joensuu, Finland; Departamento de An\'alisis Matem\'atico, Universidad de M\'alaga, Campus de
Teatinos, 29071 M\'alaga, Spain}
\email{kiansierra@hotmail.com}

\maketitle

\begin{abstract}
Let $A^p_\omega$ denote the Bergman space in the unit disc $\D$ of the complex plane induced by a radial weight~$\omega$
with the doubling property $\int_{r}^1\omega(s)\,ds\le C\int_{\frac{1+r}{2}}^1\omega(s)\,ds$.
The tent space $T^q_s(\nu,\omega)$ consists of functions such that
    \begin{equation*}
    \begin{split}
    \|f\|_{T^q_s(\nu,\omega)}^q
    =\int_\D\left(\int_{\Gamma(\zeta)}|f(z)|^s\,d\nu(z)\right)^\frac{q}s\omega(\zeta)\,dA(\zeta)
    <\infty,\quad 0<q,s<\infty.
    \end{split}
    \end{equation*}
Here $\Gamma(\zeta)$ is
a non-tangential approach region with vertex $\zeta$ in the punctured unit disc $\D\setminus\{0\}$.
We characterize the  positive Borel measures $\nu$  such that $A^p_\omega$ is embedded into the tent space $T^q_s(\nu,\omega)$, where $1+\frac{s}{p}-\frac{s}{q}>0$, by considering a generalized area operator. The results are provided in terms of Carleson measures for~$A^p_\omega$
\end{abstract}

\section{Introduction and main results}

The theory of tent spaces introduced by Coifman, Meyer and Stein~\cite{CMS}, and further studied by Cohn and Verbitsky~\cite{CohnVerbitsky} among others, shows the importance of maximal and square area functions and other objects from harmonic analysis~\cite{FefSt} in the study of Hardy spaces in the unit disc $\D=\{z:|z|<1\}$~\cite{DurenHp,Garnett1981}. The recent studies~\cite{PelRat,JAJRTent} show that tent spaces have natural analogues for Bergman spaces, and
they may play a role in the theory of weighted Bergman spaces similar to that of the original tent spaces in the Hardy space case.
The tent space $T^q_s(\nu,\om)$ consists of $\nu$-equivalence classes of
$\nu$-measurable functions $f:\D\to\C$ such that
    \begin{equation*}
    \begin{split}
    \|f\|_{T^q_s(\nu,\om)}^q&=\|A_{s,\nu}(f)\|_{L^q_\om}^q
    =\int_\D\left(\int_{\Gamma(\z)}|f(z)|^s\,d\nu(z)\right)^\frac{q}s\om(\z)\,dA(\z)
    <\infty,\quad0<q,s<\infty.
    \end{split}
    \end{equation*}
Here $\nu$ is assumed to be a positive Borel measure on $\D$, finite on compact sets, and $\om\in\DD$, that is, $\om$ is radial and $\widehat{\om}(r)=\int_r^1\om(s)\,ds$ has the doubling property $\sup_{0\le r<1}\widehat{\om}(r)/\widehat{\om}(\frac{1+r}{2})<\infty$.
Moreover,
    $$
    \Gamma(z)=\left \{  \xi \in \D : |\theta- \arg(\xi)|<\frac{1}{2}\left(1-\frac{|\xi|}{r}\right )\right  \}, \quad z=re^{i\theta}\in\overline{\D}\setminus\{0\},
    $$
are non-tangential approach regions with vertexes inside the disc, and the related tents are defined by $T(\z)=\left\{z\in\D:\z\in\Gamma(z)\right\}$ for all $\z\in\D\setminus\{0\}$. We also set $\om(T(0))=\lim_{r\to0^+}\om(T(r))$ to deal with the origin.

The purpose of this paper is three fold. First, we are interested in the question of when the weighted Bergman space $A^p_\om$, consisting of analytic functions in the unit disc $\D$ such that
    $$
    \|f\|_{A^p_\om}^p=\int_\D|f(z)|^p\om(z)\,dA(z)<\infty,\quad 0<p<\infty,
    $$
is continuously or compactly embedded into the tent space $T^q_s(\nu,\om)$.
Analogous problems for Hardy and Hardy-Sobolev spaces have been considered in~\cite{CascanteOrtega2003,Cohn,GonLouWu2010}. 
It turns out that the containment $A^p_\om\subset T^q_s(\nu,\om)$
is naturally described in terms of Carleson measures for $A^p_\om$. For $0<p,q<\infty$, a positive Borel measure $\mu$ on $\D$ is a $q$-Carleson measure for $A^p_\omega$ if there exists a constant $C>0$ such that $\|f\|_{L^q_\mu}\le C \| f \|_{A^p_\omega}$ for all $f\in A^p_\om$.

\begin{theorem}\label{Theorem1-intro}
Let $0<p,q,s<\infty$ such that $1+\frac{s}{p}-\frac{s}{q}>0$, $\om\in\DD$ and $\nu$ a positive Borel measure on $\D$, finite on compact sets, such that $\nu(\{0\})=0$. Write $\nu_\om(\z)=\om(T(\z))\,d\nu(\z)$ for all $\z\in\D$. Then the following assertions hold:
    \begin{itemize}
    \item[\rm(i)] $I_d:A^p_\omega\to T^q_s(\nu,\omega)$ is bounded if and only if $\nu_\om$ is a $(p+s-\frac{ps}{q})$-Carleson measure for~$A^p_\om$. Moreover,
    $$
    \|I_d\|^s_{A^p_\omega\to T^q_s(\nu,\omega)}\asymp\|I_d\|^{p+s-\frac{ps}{q}}_{A^p_\omega\to L^{p+s-\frac{ps}{q}}_{\nu_\om}}.
    $$
    \item[\rm(ii)] $I_d:A^p_\omega\to T^q_s(\nu,\omega)$ is compact if and only if $I_d:A^p_\omega\to L^{p+s-\frac{ps}{q}}_{\nu_\om}$ is compact.
    \end{itemize}
\end{theorem}

The requirement $\nu(\{0\})=0$, which does not carry any real restriction, is a technical hypotheses caused by the geometry of the tents $\Gamma(z)$.

Theorem~\ref{Theorem1-intro} can be interpret as a characterization of Carleson measures. This is the second aim of our study and becomes more apparent when an operator is extracted from Theorem~\ref{Theorem1-intro}. For $0<s<\infty$, the generalized area operator induced by positive measures $\mu$ and $v$  on $\D$ 
is defined by
    $$
    G^v_{\mu,s}(f)(z)=\left(\int_{\Gamma(z)}|f(\z)|^s\frac{d\mu(\z)}{v(T(\z))}\right)^\frac1s,\quad z\in\D\setminus\{0\}.
    $$
Minkowski's inequality shows that $G_{\mu,s}^v$ is sublinear if $s\ge1$. This not the case for $0<s<1$, but instead we have
    $
    \left(G_{\mu,s}^v(f+g)\right)^s\le\left(G_{\mu,s}^v(f)\right)^s+\left(G_{\mu,s}^v(g)\right)^s
    $.
Anyway, we say that $G_{\mu,s}^v:A^p_\omega\to L^q_\omega$ is bounded if there exists $C>0$ such that $\|G_{\mu,s}^v(f)\|_{L^q_\omega}\le C\|f\|_{A^p_\omega}$ for all $f\in A^p_\omega$. Write $\muov$ for the positive measure such that
    $$
    d\muov(z)=\frac{\om(T(z))}{v(T(z))}\,d\mu(z)
    $$
for $\mu$-almost every $z\in\D$. Fubini's theorem shows that
    \begin{equation}\label{Eq:case-q=s}
    \begin{split}
    \|G_{\mu,s}^v (f)\|^s_{L^{s}_\omega}
    &=\int_\D\left(\int_{\Gamma(z)}|f(\z)|^s\frac{d\mu (\z)}{v(T(\z))}\right)\omega(z)\,dA(z)\\
    &=\int_\D\left(\int_{\Gamma(z)}|f(\z)|^s\frac{d\muov (\z)}{\om(T(\z))}\right)\omega(z)\,dA(z)\\
    &=\int_{\D\setminus\{0\}}|f(\z)|^s\left(\frac{1}{\omega(T(\z))} \int_{T(\z)}\omega(z)\,dA(z)\right)d\muov(\z)\\
    &=\int_{\D\setminus\{0\}}|f(\z)|^s\,d\muov(\z)
    =\|f\|^s_{L^{s}_{\muov}}-|f(0)|^s\muov(\{0\}),
    \end{split}
    \end{equation}
and hence $G_{\mu,s}^v:A^p_\omega\to L^s_\omega$ is bounded if and only if $\muov$ is an $s$-Carleson measure for $A^p_\omega$.
For any $s>0$, we say that $G_{\mu,s}^v: A^p_\om\to L^q_\mu$ is compact if for
every bounded sequence $\{f_n\}$ in $A^p_\om$ there exists a subsequence
$\{f_{n_k}\}$ such that $G_{\mu,s}^v(f_{n_k})$ converges in $L^q(\mu)$.

The next theorem gives a characterization of Carleson measures for Bergman spaces by using the generalized area operator $G_{\mu,s}^v$. Theorem~\ref{Theorem1-intro} is a special case of this result, see also Theorem~\ref{Theorem:main-s} in Section~\ref{Sec:boundedoperators}.

\begin{theorem}\label{Theorem2-intro}
Let $0<p,q,s<\infty$ such that $s>q-p$, $\omega\in\DD$ and let $\mu,v$ be positive Borel measures on $\D$ 
such that $\mu\left(\{z\in\D:v(T(z))=0\}\right)=0=\mu(\{0\})$. Then the following assertions hold:
    \begin{itemize}
    \item[\rm(i)]$\muov$ is a $q$-Carleson measure for~$A^p_\omega$ if and only if $G_{\mu,s}^v:A^p_\omega\to L^\frac{ps}{p+s-q}_\omega$ is bounded. Moreover,
    $$
    \|G_{\mu,s}^v\|^s_{A^p_\om\to L^\frac{ps}{p+s-q}_\om}\asymp\|I_d\|^q_{A^p_\omega\to L^q_{\muov}}.
    $$
    \item[\rm(ii)] $I_d:A^p_\omega\to L^q_{\muov}$ is compact if and only if $G_{\mu,s}^v:A^p_\omega\to L^\frac{ps}{p+s-q}_\omega$ is compact.
    \end{itemize}
\end{theorem}

The third motivation of this study comes from the equivalent $A^p_\om$-norm involving square functions, given by
    \begin{equation*}
    \begin{split}
    \|f\|_{A^p_\omega}^p\asymp\int_\D\,\left(\int_{\Gamma(\z)}|D^\alpha f(z)|^2
    \left(1-\left|\frac{z}{\z}\right|\right)^{2\alpha-2}\,dA(z)\right)^{\frac{p}2}\omega(\z)\,dA(\z)+\sum_{j=0}^{[\alpha]-1}|f^{(j)}(0)|^p.
    \end{split}
    \end{equation*}
Here $0<\alpha,p<\infty$, $\om$ is a radial weight, $D^\alpha f$ denotes the fractional derivative of order $\alpha$ and $[\alpha]$ is the integer such that $[\alpha]\le \alpha<[\alpha]+1$~\cite[Theorem $4.2$]{PelRat}.
This comparability shows that the operator
    $$
    F_\alpha(f)(\z)=\left(\int_{\Gamma(\z)}|D^\alpha f(z)|^2
    \left(1-\left|\frac{z}{\z}\right|\right)^{2\alpha-2}\,dA(z)\right)^{\frac{1}2},\quad
    \z\in\D\setminus\{0\},
    $$
is bounded from $A^p_\om$ to $L^p_\om$ for each $\alpha>0$. This is no longer true when $\alpha=0$, and therefore the definition of $G_{\mu,s}^v$ is also motivated by the study of this limit case. This was the starting point in the study by Cohn on the area operator $G_{\mu}(f)(z)=\int_{\Gamma(z)}|f(\z)|\frac{d\mu (\z)}{1-|\z|}$, defined for $z$ on the boundary $\T$ of $\D$, acting from the Hardy space $H^p$ to $L^p(\T)$~\cite[Theorem~1]{Cohn}. The approach by Cohn relies on ideas by John and Nirenberg~\cite[Theorem~2.1]{Garnett1981} and Calderon-Zygmund decompositions. More recently, similar ideas together the classical factorization of Hardy spaces $H^p=H^{p_1}\cdot H^{p_2}$, $p^{-1}=p_1^{-1}+ p_2^{-1}$, were used in \cite{GonLouWu2010} to study the case $G_{\mu}:H^p\to L^q(\T)$. We do not employ these techniques in the proof of Theorem~\ref{Theorem2-intro} but instead we use a description
of the boundedness of a weighted maximal function of H\"ormander-type. To give the precise statement we need to introduce some notation. The Carleson square $S(I)$ based on an interval $I\subset\T$ is the set $S(I)=\{re^{it}\in\D:e^{it}\in I,\,
1-|I|\le r<1\}$, where $|E|$ denotes the Lebesgue measure of $E\subset\T$. We associate to each $a\in\D\setminus\{0\}$ the interval
$I_a=\{e^{i\t}:|\arg(a e^{-i\t})|\le\frac{1-|a|}{2}\}$, and denote $S(a)=S(I_a)$. For a positive Borel measure $\mu$ on~$\D$ and $\a>0$, define the weighted maximal function
    $$
    M_{\om,\alpha}(\mu)(z)=\sup_{z\in S(a)}\frac{\mu(S(a))}{\left(\om\left(S(a)
    \right)\right)^\a},\quad
    z\in\D.
    $$
In the case $\a=1$ simply write $M_{\om}(\mu)$, and if $\mu$ is of the form $\vp\om\,dA$, then $M_{\om,\a}(\mu)$ is the weighted maximal function  $M_{\om,\a}(\vp)$ of $\vp$. The following result is \cite[Theorem~3]{JAJRTent}.

\begin{lettertheorem}\label{co:maxbou}
Let $0<p\le q<\infty$ and $0<\gamma<\infty$ such that $p\gamma>1$.
Let $\om\in\DD$ and $\mu$ be a positive Borel measure on $\D$. Then
$[M_{\om}((\cdot)^{\frac{1}{\gamma}})]^{\gamma}:L^p_\omega\to
L^q(\mu)$ is bounded if and only if $M_{\om,q/p}(\mu)\in L^\infty$.
Moreover,
    $$
    \|[M_{\om}((\cdot)^{\frac{1}{\gamma}})]^{\gamma}\|^q_{\left(L^p_\om,L^q(\mu)\right)}\asymp\|M_{\om,q/p}(\mu)\|_{L^\infty}.
    $$
\end{lettertheorem}

Another maximal operator we will face is defined by $N(f)(z)=\sup_{\z\in\Gamma(z)}|f(\z)|$. It is known that $N:A^p_\om\to L^p_\om$ is bounded for each radial weight $\om$ and $\|N(f)\|_{L^p_\om}\asymp\|f\|_{A^p_\om}$ by \cite[Lemma~4.4]{PelRat}.

Theorems~\ref{Theorem1-intro} and~\ref{Theorem2-intro} would perhaps not be of much significance if we did not understand the $q$-Carleson measures for $A^p_\om$ sufficiently well. In fact, another important ingredient in the proof of our main results is the following refinement of the characterizations of Carleson measures for Bergman spaces given in \cite[Theorem~1]{JAJRTent}. The estimates for the norm of the identity operator $I_d: A^p_\om\to L^q_\mu$ are of special importance for us.

\begin{theorem}\label{Theorem:normscm}
Let $0<p,q<\infty$, $\omega\in\DD$ and let $\mu$ be a positive Borel measure on $\D$.
Further, let $dh(z)=dA(z)/(1-|z|^2)^2$ denote the hyperbolic measure.
    \begin{itemize}
    \item[\rm(i)] If $p\le q$, then $\mu$ is a $q$-Carleson measure for $A^p_\om$ if and only if
    $\sup_{a\in\D}\frac{\mu(S(a))}{\omega(S(a))^{\frac{q}{p}}}<\infty$. Moreover,
    $$
    \|I_d\|^q_{A^p_\om\to L^q_\mu}\asymp \sup_{a\in\D}\frac{\mu(S(a))}{\omega(S(a))^{\frac{q}{p}}}.
    $$
    \item[\rm(ii)] If $p\le q$, then $I_d:A^p_\om\to L^q_\mu$ is compact if and only if
    \begin{equation}\label{c1}
    \lim_{|a|\to 1^-}\frac{\mu(S(a))}{\left(\om(S(a))\right)^{q/p}}=0.
    \end{equation}
    \item[\rm(iii)] If $q<p$, then the following conditions are equivalent:
\begin{enumerate}
\item[\rm(a)] $I_d:A^p_\om\to L^q_\mu$ is compact;
\item[\rm(b)] $I_d:A^p_\om\to L^q_\mu$ is bounded;
 \item[\rm(c)] The function
    $$
    B^\om_\mu(z)=\int_{\Gamma(z)}\frac{d\mu(\z)}{\om(T(\z))},\quad
    z\in\D\setminus\{0\},
    $$
belongs to $L^{\frac{p}{p-q}}_\om$;
\item[\rm(d)] For each fixed $r\in(0,1)$, the function
    $$
    \Phi_\mu^\om(z)=\Phi_{\mu,r}^\om(z)=\int_{\Gamma(z)}\frac{\mu\left(\Delta(\z,r)\right)}{\om(T(\z))}\,dh(\z),\quad z\in\D\setminus\{0\},
    $$
belongs to $L^{\frac{p}{p-q}}_\om$;
\item[\rm(e)] For each sufficiently large $\lambda=\lambda(\om)>1$, the function
    $$
    \Psi^\om_\mu(z)=\Psi^\om_{\mu,\lambda}(z)=\int_\D \left(\frac{1-|\z|}{|1-\overline{z}\z|}\right)^\lambda\frac{d\mu(\z)}{\om(T(\z))},\quad z\in\D,
    $$
belongs to $L^{\frac{p}{p-q}}_\om$.
\end{enumerate}
Moreover,
    $$
    \|I_d\|^q_{A^p_\om\to L^q_\mu}
    \asymp\|M_\om(\mu)\|_{L^{\frac{p}{p-q}}_\om}
    \asymp\|B^\om_\mu\|_{L^{\frac{p}{p-q}}_\om}+\mu(\{0\})
    \asymp\|\Psi^\om_\mu\|_{L^{\frac{p}{p-q}}_\om}
    \asymp\|\Phi^\om_\mu\|_{L^{\frac{p}{p-q}}_\om}+\mu(\{0\}).
    $$
    \end{itemize}
\end{theorem}

An analogue of the above result for Hardy spaces is essentially known. It can be obtained by using \cite[Section~7]{Lu90}
and \cite{PelPAMS14}. Going further, the implications of the techniques used in this paper can be employed to extend the known results on the area operator on Hardy spaces \cite{Cohn,GonLouWu2010}  as well as to the study of the integral operator $T_g(f)(z)=\int_0^zf(\z)g'(\z)\,d\z$ acting on Hardy and Bergman spaces. These results are briefly discussed in Section~\ref{Sec:comments}.

\section{Carleson measures}

In this section we prove Theorem~\ref{Theorem:normscm}. For this aim we need some preliminary results and definitions.
The following lemma provides useful characterizations of weights in $\DD$. For a proof, see \cite{PelSum14,PRAntequera}.

\begin{letterlemma}\label{le:1}
Let $\om$ be a radial weight. Then the following conditions are equivalent:
\begin{itemize}
\item[\rm(i)] $\om\in\DD$;
\item[\rm(ii)] There exist $C=C(\om)>0$ and $\b=\b(\om)>0$ such that
    \begin{equation}\label{Eq:replacement-Lemma1.1}
    \begin{split}
    \widehat{\om}(r)\le C\left(\frac{1-r}{1-t}\right)^{\b}\widehat{\om}(t),\quad 0\le r\le t<1;
    \end{split}
    \end{equation}
\item[\rm(iii)] There exist $C=C(\om)>0$ and $\gamma=\gamma(\om)>0$ such that
    \begin{equation}\label{Eq:replacement-Lemma1.2}
    \begin{split}
    \int_0^t\left(\frac{1-t}{1-s}\right)^\g\om(s)\,ds
    \le C\widehat{\om}(t),\quad 0\le t<1;
    \end{split}
    \end{equation}
\item[\rm(iv)] There exists $\lambda=\lambda(\om)\ge0$ such that
    $$
    \int_\D\frac{\om(z)}{|1-\overline{\z}z|^{\lambda+1}}\,dA(z)\asymp\frac{\widehat{\om}(\zeta)}{(1-|\z|)^\lambda},\quad \z\in\D;
    $$
\item[\rm(v)] The associated weight
    $$
    \omega^\star(z)=\int_{|z|}^1\omega(s)\log\frac{s}{|z|}s\,ds,\quad z\in\D\setminus\{0\}.
    $$
satisfies
    $$
    \om(S(z))\asymp\om(T(z))\asymp\om^\star(z),\quad |z|\to1^-.
    $$
\end{itemize}
\end{letterlemma}

If $\om\in\DD$, then Lemma~\ref{le:1} shows that for each $a\in\D$ and $\gamma=\gamma(\om)>0$ large enough, the function
    $$
    F_{a,p}(z)=\left( \frac{1-|a|^2}{1-\overline{a}z}\right)^{\frac{\gamma+1}{p}},\quad z\in\D,
    $$
belongs to $A^p_\omega$ and satisfies $\|F_{a,p}\|_{A^p_\omega}^p\asymp\omega(S(a))$ and $|F_{a,p}(z)|\asymp 1$ for all $z\in S(a)$. This family of test functions will be frequently used in the sequel.

Apart from the tent spaces $T^q_s(\nu,\om)$, $0<q,s<\infty$, defined in the introduction, we will need to consider the case $q=\infty$. For $0<s<\infty$, define
    $$
    C^s_{s,\nu}(f)(\z)=\sup_{a\in \Gamma(\z)}\frac{1}{\om(T(a))}\int_{T(a)}|f(z)|^s\om(T(z))\,d\nu(z).
    $$
A quasi-norm in the tent space $T^\infty_s(\nu,\om)$ is defined by $\|f\|_{T^\infty_s(\nu,\om)}=\|C_{s,\nu}(f)\|_{L^\infty}$.

The pseudohyperbolic distance from $z$ to $w$ is defined by $\varrho(z,w)=\left|\frac{1-w}{1-\bar{z}w}\right|$, and the pseudohyperbolic
disc of center $a\in\D$ and radius $r\in(0,1)$ is denoted by $\Delta(a,r)=\{z:\varrho(a,z)<r\}$. The Euclidean discs are denoted by $D(a,r)=\{z\in\C:|a-z|<r\}$. Recall that $Z=\{z_k\}_{k=0}^\infty\subset\D$ is called a
separated sequence if it
is separated in the pseudohyperbolic metric, it is an $\e$-net if $\D=\bigcup_{k=0}^\infty \Delta(z_k,\e)$, and finally it is a
$\delta$-lattice if it is a $5\delta$-net and separated with constant $\gamma=\delta/5$. If  we have a discrete measure $\nu=\sum_k \d_{z_k}$, where $\{z_k\}$ is a separated sequence, then we write $T^p_q(\nu,\om)=T^p_q(\{z_k\},\om)$.

Recall that $I_d:A^p_{\om}\to L^q(\mu)$ is compact if it
maps bounded sets of $A^p_\om$ to relatively compact (precompact) sets of
$L^q(\mu)$. Equivalently, $I_d:A^p_{\om}\to L^q(\mu)$ is compact if and only if for
every bounded sequence $\{f_n\}$ in $A^p_\om$ there exists a subsequence that converges in $L^q(\mu)$.

\medskip

\noindent\emph{Proof of} Theorem~\ref{Theorem:normscm}. (i). There are several ways to bound the operator norm of $I_d:A^p_\om\to L^q_\mu$ from above by the claimed supremum. See \cite[Theorem~9]{JAJRTent} or \cite[Theorem~3.3]{PelSum14}, and also \cite[(11)]{JAJRTent} for the particular case $q=p$. The lower bound is obtained by using test functions, for details, see either~\cite[Lemma~8]{JAJRTent} or \cite[Theorem~3.3]{PelSum14}.

(ii). This case can be done by following the proof of \cite[Theorem~2.1(ii)]{PelRat}, with Lemma~\ref{le:1} in hand.

(iii). We first show that
    \begin{equation*}\label{ineq}
    \|I_d\|^q_{A^p_\om\to L^q_\mu}
    \le\|B^\om_\mu\|_{L^{\frac{p}{p-q}}_\om} +\mu(\{0\})
    \asymp\|\Psi^\om_\mu\|_{L^{\frac{p}{p-q}}_\om}
    \asymp\|\Phi^\om_\mu\|_{L^{\frac{p}{p-q}}_\om}+\mu(\{0\})
    \lesssim\|M_\om(\mu)\|_{L^{\frac{p}{p-q}}_\om}
    \lesssim\|I_d\|^q_{A^p_\om\to L^q_\mu}.
    \end{equation*}
Fubini's theorem, H\"older's inequality and \cite[Lemma~4.4]{PelRat} yield
    \begin{equation*}
    \begin{split}
    \|f\|^q_{L^{q}_\mu}
    &=\int_\D\left(\int_{\Gamma(z)}|f(\z)|^q \frac{d\mu(\z)}{\omega(T(\z))}\right)\omega(z)\,dA(z)+\mu(\{0\})|f(0)|^q\\
    &\le\int_\D N(f)^q(z)\int_{\Gamma(z)}\frac{d\mu(\z)}{\omega(T(\z))}\omega(z)\,dA(z)+ \mu(\{0\})|f(0)|^q\\
    &\lesssim \|N(f)\|^q_{L^{p}_\omega}\|B_\mu^\omega \|_{L^{\frac{p}{p-q}}_\omega}+\mu(\{0\})\|f\|^q_{A^{p}_\omega}
    \asymp\|f\|^q_{A^{p}_\omega}\left(\|B_\mu^\omega\|_{L^{\frac{p}{p-q}}_\omega}+\mu(\{0\}\right),
    \end{split}
    \end{equation*}
and hence $\|I_d\|^q_{A^p_\omega\to L^q_\mu}\lesssim\|B_\mu^\omega\|_{L^{\frac{p}{p-q}}_\omega}+\mu(\{0\})$. Moreover, $\|B_\mu^\omega\|_{L^{\frac{p}{p-q}}_\omega}+\mu(\{0\})\asymp\|\Psi_\mu^\omega\|_{L^{\frac{p}{p-q}}_\omega}$ by \cite[Lemma~4]{JAJRTent}.

Now write $d\nu(z)=\frac{d\mu(z)}{\om(T(z))}$ and $h_\lambda(\z,z)=\left(\frac{1-|\z|}{|1-\overline{\z}z|}\right)^\lambda$ for short. Then Fubini's theorem, \cite[Lemma~4]{JAJRTent} and the fact $\omega(T(\z))\asymp\omega(T(u))$ and $h_\lambda(\z,z)\asymp h_\lambda(u,z)$ for $\z\in\Delta(u,r)$, $r\in(0,1)$, yield
    \begin{equation}
    \begin{split}\label{eq:desc}
    \|\Psi^\omega_\mu\|^\frac{p}{p-q}_{L^\frac{p}{p-q}_\omega}
    &=\int_\D\left(\int_\D\frac{h_\lambda(\z,z)}{|\Delta(\z,r)|}|\Delta(\z,r)|\,d\nu(\z)\right)^\frac{p}{p-q}\om(z)\,dA(z)\\
    &\asymp\int_\D\left(\int_\D h_\lambda(u,z)\nu(\Delta(u,r))\,dh(u)\right)^\frac{p}{p-q}\om(z)\,dA(z)\\
    &\asymp\int_\D\left(\int_{\Gamma(z)}\nu(\Delta(u,r))\,dh(u)\right)^\frac{p}{p-q}\om(z)\,dA(z)+\mu(\{0\})\\
    &\asymp\int_\D\left(\int_{\Gamma(z)}\frac{\mu(\Delta(u,r))}{\om(T(u))}\,dh(u)\right)^\frac{p}{p-q}\om(z)\,dA(z)+\mu(\{0\})\\
    &=\|\Phi_\mu^\om\|^\frac{p}{p-q}_{L^\frac{p}{p-q}_\om}+\mu(\{0\})=\|g\|^\frac{p}{p-q}_{T^\frac{p}{p-q}_{1}(h,\omega)}+\mu(\{0\}),
    \end{split}
    \end{equation}
where $g(u)=\frac{\mu(\Delta(u,r))}{\om(T(u))}$. Now \cite[Lemma~7]{JAJRTent} implies
    \begin{equation}\label{10}
    \|g\|_{T^\frac{p}{p-q}_{1}(h,\omega)}\asymp\|C_{1,h}(g)\|_{L^\frac{p}{p-q}_\omega}.
    \end{equation}
Fubini's theorem gives
    \begin{equation*}
    \begin{split}
    C_{1,h}(g)(\z)&=\sup_{a\in\Gamma(\z)}\frac{1}{\omega(T(a))}\int_{T(a)}\mu(\Delta(z,r))\,dh(z)\\
    &=\sup_{\z\in T(a)}\frac{1}{\omega(T(a))}\int_{\D}\left(\int_{T(a)\cap\Delta(u,r)}dh(z)\right)d\mu(u).
    \end{split}
    \end{equation*}
The points $u\in\D$ for which $T(a)\cap\Delta(u,r)\ne\emptyset$ are contained in some tent $T(a')$, where $\arg a'=\arg a$ and $1-|a'|\asymp1-|a|$, for all $a\in\D\setminus D(0,\r)$, where $\r=\r(r)\in(0,1)$. Therefore
    \begin{equation*}
    \begin{split}
    \frac{1}{\omega(T(a))}\int_\D\left(\int_{T(a)\cap\Delta(u,r)}\,dh(z)\right)d\mu(u)
    &=\frac{1}{\omega(T(a))}\int_{T(a')}\left(\int_{T(a)\cap\Delta(u,r)}\,dh(z)\right)d\mu(u)\\
    &\lesssim\frac{\mu(T(a'))}{\omega(T(a))}\asymp\frac{\mu(T(a'))}{\omega(T(a'))},\quad a\in\D\setminus D(0,\r),
    \end{split}
    \end{equation*}
and it follows that
\begin{equation}\label{11}
    C_{1,h}(g)(\z)\asymp\sup_{\z\in T(a)}\frac{\mu(T(a))}{\omega(T(a))}
    \lesssim M_\omega(\mu)(\z),\quad\z\in\D\setminus\{0\}.
    \end{equation}
By combining \eqref{eq:desc}, \eqref{10} and \eqref{11},
we deduce
    \begin{equation*}
    \begin{split}
    \|\Psi^\omega_\mu\|_{L^\frac{p}{p-q}_\omega}
    &\asymp\|\Phi_\mu^\om\|_{L^\frac{p}{p-q}_\om}+\mu(\{0\})=\|g\|_{T^\frac{p}{p-q}_{1}(h,\omega)}+\mu(\{0\})\\
    &\asymp\|C_{1,h}(g)\|_{L^\frac{p}{p-q}_\omega}+\mu(\{0\})
    \lesssim\|M_\omega(\mu)\|_{L^\frac{p}{p-q}_\om}.
    \end{split}
    \end{equation*}
It remains to prove $\|M_\om(\mu)\|_{L^{\frac{p}{p-q}}_\om}\lesssim \|I_d\|^q_{A^p_\om\to L^q_\mu}$. To do this, we will show that $M_\om(\mu)$ is pointwise equivalent to the sum of two dyadic maximal functions.

Let $I_{n,k+l}=\{e^{i\theta}:\frac{\pi(k+l)}{2^{n+2}}\le\theta<\frac{\pi (k+l+1)}{2^{n+2}}\}$, and denote $\Upsilon_l=\{I_{n,k+l}:n\in\N\cup\{0\}$,  $k=0,1,\ldots,2^{n+2}-1\}$ and $l\in\{0,\frac12\}$. Define the dyadic maximal functions
    \begin{equation*}
    \begin{split}
    \widetilde{M}^d_{\omega,l}(\mu)(z)=\max\left\{\sup_{z\in T(I),I\in \Upsilon_l }\frac{\mu(T(I))}{\omega(T(I))} , \frac{\mu(\D)}{\omega(\D)}\right\},\quad z\in\D,\quad l\in\left\{0,\frac12\right\},
    \end{split}
    \end{equation*}
and set
    $$
    \widetilde{M}^d_\omega(\mu)(z)=\widetilde{M}^d_{\omega,0}(\mu)(z)+\widetilde{M}^d_{\omega,\frac{1}{2}}(\mu)(z),\quad z\in\D.
    $$
If $\om\in\DD$, then $\widetilde{M}^d_\omega(\mu)(z)\lesssim M_\omega(\mu)(z)$ for all $z\in\D$ because $\sup_{I\subset \T}\frac{\omega(S(I))}{\omega(T(I))}<\infty$ by Lemma~\ref{le:1}. For the converse inequality, given $I\subset \T$ such that $z\in S(I)$ there exist intervals $I_{n,k},I_{n,k+1}$ (if $k=2^{n+2}$, take $I_{n,k}=I_{n,0}$) such that $|I_{n+1,0}|<|I|\le|I_{n,0}|$, $I_{n,k}\cap I\neq\emptyset$ and $I_{n,k-1}\cap I=\emptyset$. We may assume that $n\ge3$, for otherwise the inequality we are searching for is immediate. Then $I\subset I_{n,k}\cup I_{n,k+1}$, and there exists $I_{n-3,m}\in\Upsilon_0\cup \Upsilon_{\frac{1}{2}}$ such that $\bigcup_{i=k-2}^{k+3}I_{n,i}\subset I_{n-3,m}$ and $S(I)\subset T(I_{n-3,m})$. Therefore Lemma~\ref{le:1} yields
    \begin{equation*}
    \begin{split}
    \frac{\mu(S(I))}{\omega(S(I))}
    \le\frac{\mu(T(I_{n-3,m}))}{\omega(T(I))} \leq \frac{\mu(T(I_{n-3,m}))}{\omega(T(I_{n+1,0}))}
    \asymp\frac{\mu(T(I_{n-3,m}))}{\omega(T(I_{n-3,m}))}\le\widetilde{M}^d_\omega(\mu)(z),\quad z\in S(I).
    \end{split}
    \end{equation*}
It follows that
    $
    M_\omega(\mu)(z)\lesssim\widetilde{M}^d_\omega(\mu)(z)
    $,
for all $z\in\D$, and hence
    \begin{equation}\label{eq:pointmax}
    M_\omega(\mu)(z)\asymp\widetilde{M}^d_\omega(\mu)(z),\quad z\in\D.
    \end{equation}

To estimate the norm of $\widetilde{M}^d_\omega(\mu)$ upwards, let choose $\{z_k\}$ be a separated sequence and
define
    $$
    S_\lambda(f)(z)=\sum_k f(z_k)\left(\frac{1-|z_k|}{1-\overline{z}_kz}\right)^\lambda,\quad z\in\D.
    $$
By \cite[Lemma~6]{JAJRTent} there exists $\lambda=\lambda(\om)>1$ such that $S_\lambda:\,T^p_{2}(\{z_k\},\omega)\to A^p_\omega$ is bounded. By denoting $\{b_{k}\}=\{f(z_{k})\}$, this implies
    \begin{equation*}
    \begin{split}
    \int_{\D}\left|\sum_{k}b_{k}h_\lambda(z_{k},z)\right|^q\,d\mu(z)
    &=\|S_\lambda(f)\|^q_{L^q(\mu)}
    \lesssim\|I_d\|^q_{A^p_\omega\to L^q_\mu}\|S_\lambda\|^q_{T^p_{2}(\{z_k\},\omega)\to A^p_\omega}\|\{b_k\}\|^q_{T^p_{2}(\{z_k\},\omega)}.
    \end{split}
    \end{equation*}
By replacing $b_k$ by $r_k(t)b_k$, where $r_k$ denotes the $k$th Rademacher function, using the fact that $|h_\lambda(z_{k},z)|\gtrsim\chi_{T(z_{k})}(z)$ for $z\in T(z_{k})$, and applying Khinchine's inequality, we deduce
    \begin{equation}\label{2}
    \begin{split}
    \int_\D \left( \sum_k |b_k|^2\chi_{T(z_k)}(z)\right)^\frac{q}{2}d\mu(z)
    \lesssim\|I_d\|^q_{A^p_\omega\to L^q_\mu}\|S_\lambda\|^q_{T^p_{2}(\{z_k\},\omega)\to A^p_\omega}\|\{b_k\}\|^q_{T^p_{2}(\{z_k\},\omega)}.
    \end{split}
    \end{equation}

Let $l\in\{0,\frac12\}$ be fixed. For each $k\in\Z$, let $\EE_k$ denote the collection of maximal dyadic tents $T\in\{T(I):I\in\Upsilon_l\}\cup\{\D\}$ with respect to inclusion such that $\mu(T)>2^k\omega(T)$, and let $E_k=\cup_{T\in\EE_k}T$. Then $2^k<\widetilde{M}^d_{\omega,l}(\mu)(z)\le2^{k+1}$ for $z \in E_k\setminus E_{k+1}$. Let now $\{b_T\}$ be a sequence indexed by $T\in\EE=\cup_{k}\EE_k$. Assume for a moment that $\mu$ has compact support. Then $\{b_T\}$ is a finite sequence. For $T\in\EE$, let $G(T)=T-\cup\{T'\in\EE:T'\subsetneq T\}$, and hence $G(T)=T-\cup\{T'\in\EE_{k+1}:T'\subsetneq T\}$ for $T\in \EE_k$. If $T_1,T_2\in\EE$ are different (either one is strictly included in the other or they are disjoint), the sets $G(T_1)$ and $G(T_2)$ are disjoint, and hence
    \begin{equation}\label{3}
    \begin{split}
    \left( \sum_{T\in \EE} |b_T|^2\chi_{T}(z)\right)^\frac{q}{2}\geq \left( \sum_{T\in \EE} |b_T|^2\chi_{G(T)}(z)\right)^\frac{q}{2} = \sum_{T\in \EE} |b_T|^q\chi_{G(T)}(z).
    \end{split}
    \end{equation}
Index the tents in $\EE$ according to which $\EE_k$ with maximal index they belong to, by writing $\EE_k\setminus\cup_{m<k}\EE_m=\{T_j^k:j\in\N\}$. Further, denote $b_{k,j}=b_{T_j^k}$ and let $z_{k,j}$ denote the vertex of $T^k_j$ i.e. $T^k_j=T(z_{k,j})$ (with the convenience that the vertex of $\D$ is the origin). The estimates \eqref{2} and \eqref{3} yield
    \begin{equation}
    \begin{split}\label{eq:4}
    &\sum_{k,j}b^q_{k,j}\left(\mu(T_j^k)-\sum_{T_i^{k+1}\subset T_j^k}\mu(T_i^{k+1})\right)
    \lesssim\|I_d\|^q_{A^p_\omega\to L^q_\mu}
    \|\{b_{k,j}\}\|^q_{T^p_{2}(\{z_{k,j}\},\omega)}\\
    &\quad=\|I_d\|^q_{A^p_\omega\to L^q_\mu}
    \left(\int_\D\left(\sum_{k,j}|b_{k,j}|^2\chi_{T_j^k}(z)\right)^\frac{p}{2}\omega(z)\,dA(z)\right)^\frac{q}{p}.
    \end{split}
    \end{equation}
Write $r=\frac{p}{q}$ for short, and choose $b_{k,j}^q=2^{k(r'-1)}$ for each $k$ and $j$. Then, by using the inequality $2^k<\widetilde{M}^d_{\omega,l}(\mu)(z)\le2^{k+1}$ for $z \in E_k\setminus E_{k+1}$, the left hand side of \eqref{eq:4} can be estimated as
    \begin{equation*}
    \begin{split}
    &\sum_{k,j}b^q_{k,j}\left(\mu(T_j^k)-\sum_{T_i^{k+1}\subset T_j^k}\mu(T_i^{k+1})\right)
    =\sum_{k}2^{k(r'-1)}\mu(E_k)-\sum_{k}2^{k(r'-1)}\sum_j\mu(T_j^k\cap E_{k+1})\\
    &=\sum_{k} (2^{k(r'-1)}-2^{(k-1)(r'-1)})\mu(E_k)
    =\left(1-\frac{1}{2^{r'-1}}\right)\sum_{k}2^{kr'}2^{-k}\sum_j\frac{\mu(T_j^k)}{\omega(T_j^k)}\omega(T_j^k)\\
    &\gtrsim\sum_k2^{kr'}\omega(E_k)
    \gtrsim\sum_{k}\int_{E_k\setminus E_{k+1}}\left(\widetilde{M}^d_{\omega,l}(\mu)(z)\right)^{r'}\omega(z)\,dA(z)
    =\|\widetilde{M}^d_{\omega,l}(\mu)\|_{L^{r'}_\omega}^{r'},
    \end{split}
    \end{equation*}
while the integral on the right hand side of \eqref{eq:4} with the notation $\eta=2^{(r'-1)\frac{2}{q}}$ becomes
    \begin{equation*}
    \begin{split}
    &\int_\D\left(\sum_{k,j}|b_{k,j}|^2\chi_{T_j^k}(z)\right)^\frac{p}{2}\omega(z)\,dA(z)
    =\int_\D\left(\sum_{k,j}\eta^k\chi_{T_j^k}(z)\right)^\frac{p}{2}\omega(z)\,dA(z)\\
    &=\int_\D\left(\sum_{k}\eta^k\chi_{E_k}(z)\right)^\frac{p}{2}\omega(z)\,dA(z)
    =\int_\D\left(\frac{\eta}{\eta-1}\sum_{k}\left(\eta^k-\eta^{k-1}\right)\chi_{E_k}(z)\right)^\frac{p}{2}\omega(z)\,dA(z)\\
    &\asymp\int_\D\left(\sum_{k}\eta^k\left(\chi_{E_k}(z)-\chi_{E_{k+1}}(z)\right)\right)^\frac{p}{2}\omega(z)\,dA(z)
    =\int_\D\left(\sum_{k}\eta^k \chi_{E_k\setminus E_{k+1} }(z)\right)^\frac{p}{2}\omega(z)\,dA(z)\\
    &=\int_\D\sum_{k} \eta^\frac{kp}{2} \chi_{E_k\setminus E_{k+1}}(z)\omega(z)\,dA(z)
    =\int_\D\sum_{k} 2^{r'k} \chi_{E_k\setminus E_{k+1} }(z)\omega(z)\,dA(z)\\
    &\asymp\sum_{k}\int_{E_k\setminus E_{k+1}}\left(\widetilde{M}^d_{\omega,l}(\mu)(z)\right)^{r'}\omega(z)\,dA(z)
    =\|\widetilde{M}^d_{\omega,l}(\mu)\|_{L^{r'}_\omega}^{r'}.
    \end{split}
    \end{equation*}
 Consequently, $\|\widetilde{M}^d_{\omega,l}(\mu)\|_{L^{r'}_\omega}^{r'}\lesssim\|I_d\|^q_{A^p_\omega\to L^q_\mu}\|\widetilde{M}^d_{\omega,l}(\mu)\|_{L^{r'}_\omega}^{\frac{r'}{r}}$, and thus $\|\widetilde{M}^d_{\omega,l}(\mu)\|_{L^{\frac{p}{p-q}}_\omega}\lesssim\|I_d\|^q_{A^p_\omega\to L^q_\mu}$ because $r'=\left(\frac{p}{q}\right)'=\frac{p}{p-q}$. Since this is valid for $l\in\{0,\frac12\}$, using Minkowski's inequality and
 \eqref{eq:pointmax} we get
    \begin{equation*}
    \begin{split}
    \|M_\omega(\mu)\|_{L^\frac{p}{p-q}_\omega}
    \asymp\|\widetilde{M}^d_\omega(\mu)\|_{L^{\frac{p}{p-q}}_\omega}
    \le\|\widetilde{M}^d_{\omega,0}(\mu)\|_{L^{\frac{p}{p-q}}_\omega}
    +\|\widetilde{M}^d_{\omega,\frac{1}{2}}(\mu)\|_{L^{\frac{p}{p-q}}_\omega}
    \lesssim\|I_d\|^q_{A^p_\omega\to L^q_\mu}
    \end{split}
    \end{equation*}
for $\mu$ with compact support. If $\mu$ is positive, then the above estimate, applied to the compactly supported $\mu_r=\chi_{D(0,r)}\mu$, and the standard limiting argument with monotone convergence theorem gives $\|M_\omega(\mu)\|_{L^\frac{p}{p-q}_\omega}\lesssim\|I_d\|^q_{A^p_\omega\to L^q_\mu}$. Hence the claimed operator norm estimates are valid and, in particular, (b)-(e) are equivalent.

To complete the proof of (iii), it suffices to show that $I_d:A^p_\om\to L^q_\mu$ is compact if (e) is satisfied.
By the hypothesis (e), \eqref{eq:desc} and
the dominated convergence theorem,
    \begin{equation}
    \begin{split}\label{s4}
    0 &=\lim_{R\to 1^-}\int_\D \left(\int_{\{R<|z|<1\}} \left(\frac{1-|z|}{|1-\overline{\z}z|}\right)^{\lambda} \frac{d\mu(z)}{\om(T(z))}\right)^{\frac{p}{p-q}}\,\om(\z)\,dA(\z)\\
    &\gtrsim \lim_{R\to 1^-} \int_\D \left(\int_{\Gamma(\z)\setminus\overline{D(0,R)}} \frac{d\mu(z)}{\om(T(z))}\right)^{\frac{p}{p-q}}\,\om(\z)\,dA(\z).
    \end{split}
    \end{equation}
Let $\{f_n\}$ be a bounded sequence in $A^p_\om$. Then $\{f_n\}$ is locally bounded and thus constitutes a normal family. Hence we may extract  a subsequence $\{f_{n_k}\}$ that converges uniformly on compact subsets of $\D$ to $f\in A^p_\om$. Write $g_k=f_{n_k}-f$. For $\e>0$, by \eqref{s4}, there exists $R_0=R_0\in(0,1)$ such that
    $$
    \int_\D\left(\int_{\Gamma(\z)\setminus\overline{D(0,R_0)}}\frac{d\mu(z)}{\om(T(z))}\right)^{\frac{p}{p-q}}\,\om(\z)\,dA(\z)<\e^{\frac{p}{p-q}}.
    $$
By the uniform convergence, we may choose $k_0\in\N$ such that $|g_k(z)|<\e^{1/q}$ for all $k\ge k_0$ and $z\in\overline{D(0,R_0)}$. Then Fubini's theorem, H\"older's inequality and \cite[Lemma~4.4]{PelRat} yield
    \begin{equation*}
    \begin{split}
    \|g_k\|^q_{L^{q}_\mu}
    &=\int_{\overline{D(0,R_0)}}|g_k(\z)|^q\,d\mu(\z)+\int_{\D\setminus\overline{D(0,R_0)}}|g_k(\z)|^q\,d\mu(\z)\\
    &\le\e\mu(\D)+\int_\D\left(\int_{\Gamma(z)\setminus\overline{D(0,R_0)}}|g_k(\z)|^q \frac{d\mu(\z)}{\omega(T(\z))}\right)\omega(z)\,dA(z)\\
    &\le\e\mu(\D)+\int_\D N(g_k)^q(z)\left(\int_{\Gamma(z)\setminus\overline{D(0,R_0)}}\frac{d\mu(\z)}{\omega(T(\z))}\right)\,\omega(z)\,dA(z)\\
    &\le\e\mu(\D)+\|N(g_k)\|^q_{L^{p}_\omega}\e
    \asymp\e\mu(\D)+\|g_k\|^q_{A^{p}_\omega}\e\lesssim \e,
    \end{split}
    \end{equation*}
and thus $I_d: A^p_\om \to L^q_\mu$ is compact.\hfill$\Box$

\section{Bounded operators $G_{\mu,s}^v:A^p_\omega\to L^q_\omega$}\label{Sec:boundedoperators}

Theorem~\ref{Theorem2-intro} is equivalent to the following result.

\begin{theorem}\label{Theorem:main-s}
Let $0<p,q,s<\infty$ such that $1+\frac{s}p-\frac{s}{q}>0$, $\omega\in\DD$ and let $\mu,v$ be  positive Borel measures on $\D$
such that $\mu\left(\{z\in\D: v(T(z))=0\}\right)=0=\mu(\{0\})$. Then the following assertions hold:
    \begin{itemize}
    \item[\rm(i)]$G_{\mu,s}^v:A^p_\omega\to L^q_\omega$ is bounded if and only if $\muov$ is a $\left(p+s-\frac{ps}{q}\right)$-Carleson measure for~$A^p_\omega$. Moreover,
    $$
    \|G_{\mu,s}^v\|^s_{A^p_\om\to L^q_\om}\asymp\|I_d\|^{p+s-\frac{ps}{q}}_{A^p_\omega\to L^{p+s-\frac{ps}{q}}_{\muov}}.
    $$
    \item[\rm(ii)] $G_{\mu,s}^v:A^p_\omega\to L^q_\omega$ is compact if and only if
$I_d:A^p_\omega\to L^{p+s-\frac{ps}{q}}_{\muov}$ is compact.
    \end{itemize}
\end{theorem}

Theorem~\ref{Theorem:main-s}(i) will be proved in two parts. We first deal with the case $q\ge p$.

\begin{theorem}\label{th:casepleq}
Let $0<p\le q<\infty$, $0<s<\infty$ and $\omega\in\DD$, and let $\mu,v$ be  positive Borel measures on $\D$
such that $\mu\left(\{z\in\D: v(T(z))=0\}\right)=0=\mu(\{0\})$. Then $G_{\mu,s}^v : A^p_\omega \to L^q_\omega$ is bounded if and only if $\muov$ is a $\left(p+s-\frac{ps}{q}\right)$-Carleson measure for $A^p_\omega$. Moreover,
    \begin{equation*}
    \begin{split}
    \|G_{\mu,s}^v\|^s_{A^p_\omega\to L^{q}_\omega}
    \asymp\|I_d\|^{p+s-\frac{ps}{q}}_{A^p_\omega\to L^{p+s-\frac{ps}{q}}_{\muov}}
    \asymp\sup_{a\in\D}\frac{\muov(S(a))}{\omega(S(a))^{1+\frac{s}{p}-\frac{s}{q}}}.
    \end{split}
    \end{equation*}
\end{theorem}

\begin{proof}
Let first $q>s$. Assume that $G_{\mu,s}^v:A^p_\omega\to L^q_\omega$ is bounded. Let $a\in\D$ and choose $\gamma=\gamma(p,q,s)$ sufficiently large so that  $\|F_{a,p}\|_{A^p_\omega}^p\asymp\omega(S(a))$ and $\|F_{a,(\frac qs)'}\|_{A^{(\frac qs)'}_\omega}^{(\frac qs)'}\asymp\omega(S(a))$. Then Fubini's theorem and H\"older's inequality yield
    \begin{equation*}
    \begin{split}
    \muov(S(a))&\asymp \int_{S(a)} |F_{a,p}(z)|^s\,d\muov(z)\\
    &\asymp\int_{S(a)}|F_{a,p}(z)|^s\left(\frac{1}{v(T(z))}\int_{T(z)}|F_{a,(\frac qs)'}(\z)|\omega(\z)\,dA(\z)\right)d\mu(z)\\
    &\lesssim\int_\D|F_{a,(\frac qs)'}(\z)|\left(\int_{\Gamma(\z)}|F_{a,p}(z)|^s\frac{d\mu(z)}{v(T(z))}\right)\omega(\z)\,dA(\z)\\
    &\le\|F_{a,(\frac qs)'}\|_{A^{(\frac qs)'}_\omega}\|G_{\mu,s}^v (F_{a,p})\|^s_{L^{q}_\omega}
    \le\|F_{a,(\frac qs)'}\|_{A^{(\frac qs)'}_\omega}\|G_{\mu,s}^v\|^s_{A^p_\omega\to L^{q}_\omega}\|F_{a,p}\|^s_{A^p_\omega}\\
    &\asymp\|G_{\mu,s}^v\|^s_{A^p_\omega\to L^{q}_\omega}\omega(S(a))^{\frac{1}{(\frac qs)'}} \omega(S(a))^{\frac{s}{p}}
    \asymp\|G_{\mu,s}^v\|^s_{A^p_\omega\to L^{q}_\omega} \omega(S(a))^{1+\frac{s}{p}-\frac{s}{q}}.
    \end{split}
    \end{equation*}
Hence
    $
    \sup_{a\in\D}\frac{\muov(S(a))}{\omega(S(a))^{1+\frac{s}{p}-\frac{s}{q}}}\lesssim \|G_{\mu,s}^v\|^s_{A^p_\omega\to L^{q}_\omega}
    $
and $\muov$ is a $\left(p+s-\frac{ps}{q}\right)$-Carleson measure for $A^p_\omega$ by Theorem~\ref{Theorem:normscm}.

Conversely, let $q>s$ and $\muov$ be a $\left(p+s-\frac{ps}{q}\right)$-Carleson measure for $A^p_\omega$. Write $t=t(p,q,s)=p+s-\frac{ps}{q}>s$ for short. Since $\frac{t}{p}=(\frac{t}{s})'/(\frac{q}{s})'$, \cite[Theorem~3]{JAJRTent} shows that $M_\omega:L^{(\frac{q}{s})'}_\omega \to L^{(\frac{t}{s})'}_{\muov}$ is bounded with $\|M_\om\|^{(\frac{t}{s})'}_{L^{(\frac{q}{s})'}_\omega\to L^{(\frac{t}{s})'}_{\muov}}\asymp \sup_{a\in\D}\frac{\muov(S(a))}{\omega(S(a))^{t/p}}$.
 This together with Theorem~\ref{Theorem:normscm}, Fubini's theorem and H\"older's inequality give
    \begin{equation*}
    \begin{split}
    \|G_{\mu,s}^v(f)\|^s_{L^{q}_\omega}&=\sup_{\|h\|_{L^{(\frac{q}{s})'}_\omega}\le1}\int_\D|h(z)|
    \left(\int_{\Gamma(z)}|f(\z)|^s\frac{d\muov (\z)}{\omega(T(\z))}\right)\omega(z)\,dA(z)\\
    &=\sup_{\| h\|_{L^{(\frac{q}{s})'}_\omega}\le1}\int_\D|f(\z)|^s\left(\frac{1}{\omega(T(\z))}\int_{T(\z)}|h(z)|\omega(z)\,dA(z)\right)d\muov(\z)\\
    &\lesssim\sup_{\| h\|_{L^{(\frac{q}{s})'}_\omega}\le1}\int_\D|f(\z)|^sM_\omega(h)(\z)\,d\muov(\z)
    \le\sup_{\|h\|_{L^{(\frac{q}{s})'}_\omega}\le1}\|f\|^s_{L^{t}_{\muov}}\| M_\omega(h)\|_{L^{(\frac{t}{s})'}_{\muov}}\\
    &\le\sup_{\|h\|_{L^{(\frac{q}{s})'}_\omega}\le1}\|I_d\|^s_{A^p_\omega\to L^{t}_{\muov}}\|f\|^s_{A^{p}_\omega}
    \|M_\om\|_{L^{(\frac{q}{s})'}_\omega\to L^{(\frac{t}{s})'}_{\muov}}\|h\|_{L^{(\frac{q}{s})'}_\omega}\\
    &\asymp\left(\sup_{a\in\D}\frac{\muov(S(a))}{\omega(S(a))^{1+\frac{s}{p}-\frac{s}{q}}}\right)^{\frac{s}t+\frac1{(\frac{t}{s})'}}\|f\|^s_{A^{p}_\omega},
    \end{split}
    \end{equation*}
and hence $G_{\mu,s}^v:A^p_\omega\to L^q_\omega$ is bounded and $\|G_{\mu,s}^v(f)\|^s_{L^{q}_\omega}\lesssim \sup_{a\in\D}\frac{\muov(S(a))}{\omega(S(a))^{1+\frac{s}{p}-\frac{s}{q}}}$.

Since the assertion is valid for $q=s$ by \eqref{Eq:case-q=s}, it remains to consider the case $q<s$. Let first $\muov$ be a $\left(p+s-\frac{ps}{q}\right)$-Carleson measure for $A^p_\omega$, and let $0<x<s$. H\"older's inequality and Fubini's theorem yield
    \begin{equation*}
    \begin{split}
    \|G_{\mu,s}^v(f)\|^q_{L^{q}_\omega}
    &=\int_\D\left(\int_{\Gamma(z)}|f(\z)|^{x+s-x}\frac{d\muov(\z)}{\omega(T(\z))}\right)^\frac{q}{s} \omega(z)\,dA(z)\\
    &\le\int_\D N(f)(z)^{\frac{qx}{s}}\left(\int_{\Gamma(z)}|f(\z)|^{s-x}\frac{d\muov(\z)}{\omega(T(\z))}\right)^\frac{q}{s}\omega(z)\,dA(z)\\
    &\le\left(\int_\D N(f)(z)^{\frac{qx}{s-q}}\omega(z)\,dA(z)\right)^{1-\frac{q}{s}}
    \left(\int_\D\int_{\Gamma(z)}|f(\z)|^{s-x}\frac{d\muov(\z)}{\omega(T(\z))}\omega(z)\,dA(z)\right)^\frac{q}{s}\\
    &\lesssim\|N(f)\|^{\frac{qx}{s}}_{L_\omega^{\frac{qx}{s-q}}}\left(\int_\D|f(\z)|^{s-x}\,d\muov (\z)\right)^\frac{q}{s}
    \asymp\|N(f)\|^{\frac{qx}{s}}_{L_\omega^{\frac{qx}{s-q}}}\|f\|^{\frac{q(s-x)}{s}}_{L_{\muov}^{(s-x)}}.
    \end{split}
    \end{equation*}
Take $x=\frac{p(s-q)}{q}<s$ so that $s-x=s+p-\frac{ps}{q}$. Then the estimates above together with \cite[Lemma 4.4]{PelRat} and Theorem~\ref{Theorem:normscm} give
    \begin{equation*}
    \begin{split}
    \|G_{\mu,s}^v(f)\|^q_{L^{q}_\omega}
    \lesssim\|f\|^{\frac{qx}{s}}_{A_\omega^{p}}\|I_d\|^{\frac{q(s-x)}{s}}_{A_\omega^{p}\to L_{\muov}^{s+p-\frac{ps}{q}}}\|f\|^{\frac{q(s-x)}{s}}_{A_\omega^{p}}
    \asymp\left(\sup_{a\in\D}\frac{\muov(S(a))}{\omega(S(a))^{1+\frac{s}{p}-\frac{s}{q}}}\right)^\frac{q}{s}\|f\|^{q}_{A_\omega^{p}},
    \end{split}
    \end{equation*}
and hence $G_{\mu,s}^v:A^p_\omega\to L^q_\omega$ is bounded with $\|G_\mu^v\|^s_{A^p_\omega\to L^{q}_\omega}\lesssim \sup_{a\in\D}\frac{\muov(S(a))}{\omega(S(a))^{1+\frac{s}{p}-\frac{s}{q}}}$.

Conversely, let $q<s$ and $G_{\mu,s}^v:A^p_\omega\to L^q_\omega$ be bounded. Choose $\alpha>\beta>1$ such that $\frac{\beta}{\alpha}=\frac{q}{s}$. Fubini's theorem and H\"older's inequality  yield
    \begin{equation}\label{1-s}
    \begin{split}
    \muov(S(a))&=\int_\D\chi_{S(a)}(z)\,d\muov(z)
    \asymp\int_\D\chi_{S(a)}(z)|F_{a,p}(z)|^s\frac{1}{v(T(z))}\int_{T(z)}\omega(\z)\,dA(\z)\,d\mu(z)\\
    &=\int_\D\left(\int_{\Gamma(\z)}\chi_{S(a)}(z)|F_{a,p}(z)|^s\frac{ d\mu (z)}{v(T(z))}\right)^{\frac{1}{\alpha}+\frac{1}{\alpha'}}\omega(\z)\,dA(\z)\\
    &\le\left(\int_\D \left(\int_{\Gamma(\z)}|F_{a,p}(z)|^s\frac{d\mu(z)}{v(T(z))}\right)^{\frac{\beta}{\alpha}}\omega(\z)\,dA(\z)\right)^{\frac{1}{\beta}}\\
    &\quad\cdot\left(\int_\D\left(\int_{\Gamma(\z)}\chi_{S(a)}(z)\frac{d\mu(z)}{v(T(z))}\right)^{\frac{\beta'}{\alpha'}}
    \omega(\z)\,dA(\z)\right)^{\frac{1}{\beta'}} \\
    &=\|G_{\mu,s}^v(F_{a,p})\|^{\frac{q}{\beta}}_{L^{q}_\omega}\|G_{\mu,1}^v(\chi_{S(a)})\|^{\frac{1}{\alpha'}}_{L^{\frac{\beta'}{\alpha'}}_\omega}.
    \end{split}
    \end{equation}
Now $\frac{\beta'}{\alpha'}>1$ because $\beta<\alpha$, and hence $\|G_{\mu,1}^v( \chi_{S(a)})\|_{L^{\frac{\beta'}{\alpha'}}_\omega}$ can be estimated by duality arguments. Namely, since $\left(\frac{\beta'}{\alpha'}\right)'=\frac{\beta(\alpha-1)}{\alpha-\beta}$, Fubini's theorem, H\"older's inequality and \cite[Theorem~3]{JAJRTent} give
    \begin{equation}\label{1-s'}
    \begin{split}
    \|G_{\mu,1}^v(\chi_{S(a)})\|_{L^{\frac{\beta'}{\alpha'}}_\omega}
    &=\sup_{\|h\|_{L^{\frac{\beta(\alpha-1)}{\alpha-\beta}}_\omega}\le1}\int_\D|h(\z)|G_{\mu,1}^v(\chi_{S(a)})(\z) \omega(\z)\,dA(\z)\\
    &=\sup_{\|h\|_{L^{\frac{\beta(\alpha-1)}{\alpha-\beta}}_\omega}\le1}\int_\D|h(\z)|\left(\int_{\Gamma(\z)}\chi_{S(a)}(z)\frac{d\mu (z)}{v(T(z))} \right) \omega(\z)\,dA(\z)\\
    &=\sup_{\|h\|_{L^{\frac{\beta(\alpha-1)}{\alpha-\beta}}_\omega}\le1}\int_\D\chi_{S(a)}(z)
    \left(\frac{1}{v(T(z))}\int_{T(z)}|h(\z)|\omega(\z)\,dA(\z)\right)d\mu(z)\\
    &=\sup_{\|h\|_{L^{\frac{\beta(\alpha-1)}{\alpha-\beta}}_\omega}\le1}\int_\D\chi_{S(a)}(z)
    \left(\frac{1}{\om(T(z))}\int_{T(z)}|h(\z)|\omega(\z)\,dA(\z)\right)d\muov(z)\\
    &\lesssim\sup_{\|h\|_{L^{\frac{\beta(\alpha-1)}{\alpha-\beta}}_\omega}\le1}\int_\D\chi_{S(a)}(z)M_\omega(h)(z)\,d\muov(z)\\
    &\le\sup_{\|h\|_{L^{\frac{\beta(\alpha-1)}{\alpha-\beta}}_\omega}\le1}\left(\int_\D\chi_{S(a)}(z)\,d\muov(z)\right)^{\frac{\alpha'}{\beta'}}\\
    &\quad\cdot\left(\int_\D M_\omega(h)(z)^\frac{\beta(\alpha-1)}{\alpha-\beta}\chi_{S(a)}(z)\,d\muov(z)\right)^{1-\frac{\alpha'}{\beta'}}\\
    &\lesssim\muov(S(a))^{\frac{\alpha'}{\beta'}} \left(\sup_{z\in \D} \frac{\muov(S(z)\cap S(a))}{\omega(S(z))} \right)^{1-\frac{\alpha'}{\beta'}},\quad a\in\D.
    \end{split}
    \end{equation}
By combining this with \eqref{1-s} and using the norm estimate $\|F_{a,p}\|_{A^p_\omega}^p\asymp\omega(S(a))$, we deduce
    \begin{equation*}
    \begin{split}
    \muov (S(a))
    &\lesssim\|G_{\mu,s}^v(F_{a,p})\|^{\frac{q}{\beta}}_{L^{q}_\omega}
    \|G_{\mu,1}^v(\chi_{S(a)})\|^{\frac{1}{\alpha'}}_{L^{\frac{\beta'}{\alpha'}}_\omega}\\
    &\lesssim\|G_{\mu,s}^v\|^{\frac{q}{\beta}}_{A^p_\omega\to L^{q}_\omega}\omega(S(a))^{\frac{q}{p\beta}} \muov(S(a))^{\frac{1}{\beta'}} \left(\sup_{z\in \D} \frac{\muov(S(z)\cap S(a))}{\omega(S(z))} \right)^{\frac{1}{\alpha'}-\frac{1}{\beta'}},
    \end{split}
    \end{equation*}
which yields
    \begin{equation}
    \begin{split}\label{eq:1-s}
    \muov (S(a))^{\frac{1}{\beta}}\omega(S(a))^{-\frac{q}{p\b}}
    &\lesssim\|G_{\mu,s}^v\|^{\frac{q}{\b}}_{A^p_\omega\to L^{q}_\omega} \left(\sup_{z\in \D} \frac{\muov(S(z)\cap S(a))}{\omega(S(z))} \right)^{\frac{1}{\alpha'}-\frac{1}{\beta'}},\quad a\in\D.
    \end{split}
    \end{equation}

Define $d\mu_r(z)=\chi_{D(0,r)}(z)d\mu(z)$ for $0<r<1$. Then
    \begin{equation}
    \begin{split}\label{eq:2-s}
    \left(G_{\mu_r,s}^v(f)(z)\right)^s
    &=\int_{\Gamma(z)}|f(\z)|^s\frac{d\mu_r (\z)}{v(T(\z))}
    =\int_{\Gamma(z)\cap D(0,r)}|f(\z)|^s\frac{d\mu (\z)}{v(T(\z))}\\
    &\le\int_{\Gamma(z)}|f(\z)|^s\frac{d\mu (\z)}{v(T(\z))}
    =\left(G_{\mu,s}^v(f)(z)\right)^s,\quad z\in\D\setminus\{0\},
    \end{split}
    \end{equation}
and hence $\|G_{\mu_r,s}^v\|_{A^p_\omega\to L^{q}_\omega}\le\|G_{\mu,s}^v\|_{A^p_\omega\to L^{q}_\omega}$ for all $0<r<1$.

If $q=p$, then \eqref{eq:1-s} applied to $\mu_r$ implies
    \begin{equation*}
    \begin{split}
    (\muov)_r(S(a))^{\frac{1}{\beta}}\omega(S(a))^{-\frac{1}{\beta}}
    &\lesssim\|G_{\mu_r,s}^v\|^{\frac{p}{\beta}}_{A^p_\omega\to L^{p}_\omega}\left(\sup_{z\in\D}\frac{(\muov)_r(S(z))}{\omega(S(z))} \right)^{\frac{1}{\alpha'}-\frac{1}{\beta'}},\quad a\in\D,
    \end{split}
    \end{equation*}
and hence
    \begin{equation*}
    \begin{split}
    \left(\sup_{z\in\D}\frac{(\muov)_r(S(z))}{\omega(S(z))} \right)^{\frac{1}{\beta}}
    &\lesssim\|G_{\mu,s}^v\|^{\frac{p}{\beta}}_{A^p_\omega\to L^{p}_\omega}\left(\sup_{z\in\D}\frac{(\muov)_r(S(z))}{\omega(S(z))}\right)^{\frac{1}{\alpha'}-\frac{1}{\beta'}}.
    \end{split}
    \end{equation*}
Consequently,
    \begin{equation*}
    \begin{split}
    \sup_{z\in\D, r\in (0,1)}\frac{(\muov)_r(S(z))}{\omega(S(z))}
    &\lesssim\|G_{\mu,s}^v\|^s_{A^p_\omega\to L^{p}_\omega}.
    \end{split}
    \end{equation*}
So Fatou's lemma and Theorem~\ref{Theorem:normscm} show that $\muov$ is a $p$-Carleson measure for $A^p_\omega$ with
    $$
    \|I_d\|^p_{A^p_\om\to L^p_{\muov}}
    \asymp\sup_{a\in\D}\frac{\muov(S(a))}{\omega(S(a))}
    \lesssim\|G_{\mu,s}^v\|^s_{A^p_\omega\to L^{p}_\omega}.
    $$

If $q>p$, then applying \eqref{eq:1-s}  to $\mu_r$ and bearing in mind that $\om$ is radial
    \begin{equation*}
    \begin{split}
    \frac{(\muov)_r(S(a))}{\omega(S(a))^{1+\frac{s}{p}-\frac{s}{q}}}
    &\lesssim\|G_{\mu_r,s}^v\|^q_{A^p_\omega\to L^{q}_\omega}\omega(S(a))^{\frac{q}{p}-1-\frac{s}{p}+\frac{s}{q}}
    \left(\sup_{z\in\D}\frac{(\muov)_r(S(z)\cap S(a))}{\omega(S(z))}\right)^{\frac{\beta}{\alpha'}-\frac{\beta}{\beta'}}\\
    &=\|G_{\mu_r,s}^v\|^q_{A^p_\omega\to L^{q}_\omega}\omega(S(a))^{\frac{-(q-p)(s-q)}{pq}}
    \left(\sup_{z:S(z)\subset S(a)}\frac{(\muov)_r(S(z)\cap S(a))}{\omega(S(z))}\right)^{1-\frac{q}{s}}\\
    &=\|G_{\mu_r,s}^v\|^q_{A^p_\omega\to L^{q}_\omega}
    \left(\sup_{z:S(z)\subset S(a)}\frac{(\muov)_r(S(z)\cap S(a))}{\omega(S(a))^{s\frac{q-p}{pq}}\omega(S(z))}\right)^{1-\frac{q}{s}}\\
    &\le\|G_{\mu_r,s}^v\|^q_{A^p_\omega \to L^{q}_\omega}
    \left(\sup_{z:S(z)\subset S(a)}\frac{(\muov)_r(S(z))}{\omega(S(z))^{1+\frac{s}{p}-\frac{s}{q}}}\right)^{1-\frac{q}{s}},\quad a\in\D.
    \end{split}
    \end{equation*}
Consequently,
    \begin{equation*}
    \begin{split}
    \sup_{a\in\D}\left(\frac{(\muov)_r(S(a))}{\omega(S(a))^{1+\frac{s}{p}-\frac{s}{q}}}\right)
    &\lesssim\|G_{\mu,s}^v\|^q_{A^p_\omega\to L^{q}_\omega}
    \left(\sup_{a\in\D}\sup_{z:S(z)\subset S(a)}\frac{(\muov)_r(S(z))}{\omega(S(z))^{1+\frac{s}{p}-\frac{s}{q}}}\right)^{1-\frac{q}{s}}\\
    &=\|G_{\mu,s}^v\|^q_{A^p_\omega\to L^{q}_\omega}\left(\sup_{a\in \D}\frac{(\muov)_r(S(a))}{\omega(S(a))^{1+\frac{s}{p}-\frac{s}{q}}} \right)^{1-\frac{q}{s}},
    \end{split}
    \end{equation*}
and thus
    \begin{equation*}
    \begin{split}
    \sup_{a\in\D,\,r\in(0,1)}\frac{(\muov)_r(S(a))}{\omega(S(a))^{1+\frac{s}{p}-\frac{s}{q}}}
    &\lesssim\|G_{\mu,s}^v\|^s_{A^p_\omega\to L^{q}_\omega}.
    \end{split}
    \end{equation*}
Fatou's lemma and Theorem~\ref{Theorem:normscm} show that $\muov$ is a $\left(p+s-\frac{ps}{q}\right)$-Carleson measure for~$A^p_\omega$ with the corresponding inequality of norms.
\end{proof}

\begin{theorem}\label{fukifuki}
Let $0<q<p<\infty$ and $0<s<\infty$ such that $1+\frac{s}p-\frac{s}{q}>0$, $\omega\in\DD$ and let $\mu,v$ be  positive Borel measures on $\D$
such that $\mu\left(\{z\in\D: v(T(z))=0\}\right)=0=\mu(\{0\})$. Then $G_{\mu,s}^v:A^p_\omega\to L^q_\omega$ is bounded if and only if $\muov$ is a $p\left(1+\frac{s}p-\frac{s}{q}\right)$-Carleson measure for $A^p_\omega$. Moreover,
    \begin{equation*}
    \begin{split}
    \|G_{\mu,s}^v\|^s_{A^p_\omega\to L^{q}_\omega}
    \asymp\|I_d\|^{p+s-\frac{ps}{q}}_{A^p_\omega\to L^{p+s-\frac{ps}{q}}_{\muov}}
    \asymp\|B_\mu^v\|_{L^{\frac{qp}{s(p-q)}}_\omega}.
    \end{split}
    \end{equation*}
\end{theorem}

\begin{proof}
The equivalence $\|I_d\|^{p+s-\frac{ps}{q}}_{A^p_\omega\to L^{p+s-\frac{ps}{q}}_{\muov}}
    \asymp\|B_\mu^v\|_{L^{\frac{qp}{s(p-q)}}_\omega}$ follows from Theorem~\ref{Theorem:normscm}.
\par If $B_\mu^v\in L^{\frac{qp}{s(p-q)}}_\omega$, then H\"older's inequality and \cite[Lemma~4.4]{PelRat} give
    \begin{equation*}
    \begin{split}
    \|G_{\mu,s}^v(f)\|^q_{L^{q}_\omega}
    &=\int_\D\left(\int_{\Gamma(z)}|f(\z)|^s\frac{d\mu(\z)}{v(T(\z))}\right)^\frac{q}{s}\omega(z)\,dA(z)\\
    &\le\int_\D N(f)^q(z)\left(\int_{\Gamma(z)}\frac{d\mu(\z)}{v(T(\z))}\right)^\frac{q}{s}\omega(z)\,dA(z)\\
    &\le\|N(f) \|^q_{L^{p}_\omega}\|B_\mu^v\|^\frac{q}{s}_{L^{\frac{qp}{s(p-q)}}_\omega}
    \asymp\|f\|^q_{A^{p}_\omega}\|B_\mu^v\|^\frac{q}{s}_{L^{\frac{qp}{s(p-q)}}_\omega},
    \end{split}
    \end{equation*}
and hence $G_{\mu,s}^v:A^p_\omega\to L^q_\omega$ is bounded and $\|G_{\mu,s}^v\|^s_{A^p_\omega\to L^{q}_\omega}\lesssim\|B_\mu^v\|_{L^{\frac{qp}{p-q}}_\omega}$.

Assume now that $G_{\mu,s}^v:A^p_\omega\to L^q_\omega$ is bounded, and let first $q>s$ and write $t=t(p,q,s)=(p+s-\frac{ps}{q})=s+p(1-\frac{s}{q})$. It suffices to show that $\muov$ is a $t$-Carleson measure for $A^p_\omega$. Fubini's theorem, H\"older's inequality and \cite[Lemma~4.4]{PelRat} yield
    \begin{equation*}
    \begin{split}
    \|f\|_{L^t_{\muov}}^t
    &=\int_\D|f(z)|^t\left(\int_{T(z)}\om(\z)\,dA(\z)\right)\frac{d\mu(z)}{v(T(z))}\\
    &=\int_\D\left(\int_{\Gamma(\z)}|f(z)|^{s+p(1-\frac{s}{q})}\frac{d\mu(z)}{v(T(z))}\right)\om(\z)\,dA(\z)\\
    &\le\int_\D N(f)^{p(1-\frac{s}{q})}(\z)\left(\int_{\Gamma(\z)}|f(z)|^s\frac{d\mu(z)}{v(T(z))}\right)\om(\z)\,dA(\z)\\
    &\lesssim\|N(f)\|_{L^p_\om}^{p(1-\frac{s}{q})}\|G^v_{\mu,s}(f)\|^s_{L^q_\om}
    \asymp\|f\|_{A^p_\om}^{p(1-\frac{s}{q})}\|G^v_{\mu,s}(f)\|^s_{L^q_\om}\\
    &\le \|G^v_\mu\|^s_{A^p_\om\to L^q_\om}\|f\|_{A^p_\om}^t
    \end{split}
    \end{equation*}
and hence $I_d:\,A^p_\om\to L^t_{\muov}$ is bounded with
 $\|I_d\|^t_{A^p_\om\to L^t_{\muov}}\lesssim\|G^v_\mu\|^s_{A^p_\om\to L^q_\om}$.

If $q=s$, then $t=s$ and the result follows from \eqref{Eq:case-q=s}.
\par If $q<s$, then $t=p+s-\frac{ps}{q}<s$ and $\frac{q}{t}>1$ because $\min\{p,s\}>q$. Hence H\"older's inequality gives
    \begin{equation*}
    \begin{split}
    \|f\|^t_{L^t_{\muov}}
    &=\int_\D\left(\int_{\Gamma(\z)}|f(z)|^t\frac{d\mu(z)}{v(T(z))}\right)\omega(\z)\,dA(\z)\\
    &\lesssim\int_\D\left(\int_{\Gamma(\z)}|f(z)|^s\frac{d\mu(z)}{v(T(z))}\right)^\frac{t}{s}
    \left(\int_{\Gamma(\z)}\frac{d\mu(z)}{v(T(z))}\right)^{1-\frac{t}{s}}\omega(\z)\,dA(\z)\\
    &\le\left(\int_\D\left(\int_{\Gamma(\z)}|f(z)|^s\frac{d\mu(z)}{v(T(z))}\right)^\frac{q}{s}\omega(\z)\,dA(\z)\right)^\frac{t}{q} \\&\cdot\left(\int_\D\left(\int_{\Gamma(\z)}\frac{d\mu(z)}{v(T(z))}\right)^{\frac{s-t}{s}\frac{q}{q-t}}\omega(\z)\,dA(\z)\right)^\frac{q-t}{q}\\
    &=\|G_{\mu,s}^v(f)\|_{L^q_\omega}^t\|B_\mu^v\|_{L^\frac{qp}{s(p-q)}_\omega}^{1-\frac{t}{s}}\\
    &\le\|G_{\mu,s}^v\|_{A^p_\omega\to L^q_\omega}^t\|B_{\mu,s}^v\|_{L^\frac{qp}{s(p-q)}_\omega}^{1-\frac{t}{s}} \|f\|_{A^p_\omega}^t.
    \end{split}
    \end{equation*}
This applied to $\mu_r$ yields
    \begin{equation*}
    \begin{split}
    \left(\frac{\|f\|_{L^{t}_{(\muov)_r}}}{\| f \|_{A^p_\omega}}\right)^t
    \lesssim\|G_{\mu_r,s}^v\|_{A^p_\omega\to L^q_\omega}^t
    \|B_{\mu_r}^v\|_{L^{\frac{qp}{s(p-q)}}_\omega}^{1-\frac{t}{s}}  ,\quad f\in A^p_\omega,\quad f\not\equiv0,
    \end{split}
    \end{equation*}
and hence
    \begin{equation*}
    \begin{split}
    \|I_d\|_{ A^p_\omega\to L^t_{{(\muov)_r}}}^t
    \le\|G_{\mu_r,s}^v\|_{A^p_\omega\to L^q_\omega}^t\|B_{\mu_r}^v\|_{L^{\frac{qp}{s(p-q)}}_\omega}^{1-\frac{t}{s}}.
    \end{split}
    \end{equation*}
Since
    $
    \|I_d\|_{A^p_\omega\to L^t_{({\muov})_r}}^t\asymp\|B_{\mu_r}^v\|_{L^\frac{qp}{s(p-q)}_\omega}
    $
by Theorem~\ref{Theorem:normscm}, we deduce
    \begin{equation*}
    \begin{split}
    \|B_{\mu_r}^v\|_{L^\frac{qp}{s(p-q)}_\omega}
    \lesssim\| G_{\mu_r,s}^v \|_{ A^p_\omega \to L^q_\omega}^t\|B_{\mu_r}^v\|_{L^{\frac{qp}{s(p-q)}}_\omega}^{1-\frac{t}{s}},
    \end{split}
    \end{equation*}
and hence
    \begin{equation*}
    \begin{split}
    \|B_{\mu_r}^v\|_{L^\frac{qp}{s(p-q)}_\omega}
    \lesssim\|G_{\mu_r,s}^v\|^s_{A^p_\omega\to L^q_\omega}
    \le\|G_{\mu,s}^v\|^s_{A^p_\omega\to L^q_\omega},
    \end{split}
    \end{equation*}
which together with Fatou's lemma gives
    $$
    \|B_{\mu}^v\|_{L^\frac{qp}{s(p-q)}_\omega}
    \le\liminf_{r\to 1}\|B_{\mu_r}^v\|_{L^\frac{qp}{s(p-q)}_\omega}
    \lesssim\|G_{\mu,s}^v\|^s_{A^p_\omega\to L^q_\omega}.
    $$
This finishes the proof.
\end{proof}

\section{Compact operators $G_{\mu,s}^v:A^p_\omega\to L^q_\omega$}

In this section we prove Theorem~\ref{Theorem:main-s}(ii).

\begin{lemma}\label{technichal}
Let $\nu$ be a finite positive  Borel measure on $\D$ and $0<p<\infty$. If $\{\vp_n\}_{n=0}^\infty\subset L^p_\nu$ and $\vp\in  L^p_\nu$ satisfy
 $\lim_{n\to\infty}\|\vp_n\|_{L^p_\nu}=\|\vp\|_{L^p_\nu}$ and
$\lim_{n\to\infty}\vp_n(z)=\vp(z)$ $\nu$-a.e. on $\D$,
then $\lim_{n\to\infty}\|\vp_n-\vp\|_{L^p_\nu}=0$.
\end{lemma}

\begin{proof}
See the proof of \cite[Lemma~1~p.~21]{DurenHp}.
\end{proof}

\begin{theorem}\label{th:c1}
Let $0<p\le q<\infty$, $0<s<\infty$ and $\omega\in\DD$, and let $\mu,v$ be  positive Borel measures on $\D$
such that $\mu\left(\{z\in\D: v(T(z))=0\}\right)=0=\mu(\{0\})$. Then $G_{\mu,s}^v : A^p_\omega \to L^q_\omega$ is compact if and only if
$I_d:\,A^p_\omega\to L^{p+s-\frac{ps}{q}}_{\muov}$ is compact.
\end{theorem}

\begin{proof}
Let first $q>s$. Assume that $G_{\mu,s}^v:A^p_\omega\to L^q_\omega$ is compact.
For each $a\in\D$, let
    $
    f_{a,p}(z)
    =\left(\om(S(a))\right)^{-\frac1p}F_{a,p}(z)
    =\left(\om(S(a))\right)^{-\frac1p}\left(\frac{1-|a|^2}{1-\overline{a}z}\right)^{\frac{\g+1}{p}}
    $.
By Lemma~\ref{le:1} we may choose $\gamma=\gamma(p,q,s,\om)$ sufficiently large such that $\sup_{a\in\D}\|f_{a,p}\|_{A^p_\omega}^p\asymp1$,
$\sup_{a\in\D}\|f_{a,(\frac qs)'}\|_{A^{(\frac qs)'}_\omega}^{(\frac qs)'}\asymp1$ and $f_{a,p}$ converges uniformly to zero on compact subsets of $\D$ as $|a|\to 1^-$.
A standard argument shows that
    \begin{equation}\label{eq:c1}
    \lim_{|a|\to 1^-}\|G_{\mu,s}^v (f_{a,p})\|_{L^{q}_\omega}=0.
    \end{equation}
Fubini's theorem and H\"older's inequality yield
    \begin{equation*}
    \begin{split}
    \frac{\muov(S(a))}{\omega(S(a))^{1+\frac{s}{p}-\frac{s}{q}}}&\asymp
    \frac{1}{\omega(S(a))^{1-\frac{s}{q}}}\int_{S(a)} |f_{a,p}(z)|^s\,d\muov(z)\\
    &\asymp\int_{S(a)}|f_{a,p}(z)|^s\left(\frac{1}{v(T(z))}\int_{T(z)}|f_{a,(\frac qs)'}(\z)|\omega(\z)\,dA(\z)\right)d\mu(z)\\
    &\lesssim\int_\D|f_{a,(\frac qs)'}(\z)|\left(\int_{\Gamma(\z)}|f_{a,p}(z)|^s\frac{d\mu(z)}{v(T(z))}\right)\omega(\z)\,dA(\z)\\
    &\le\|f_{a,(\frac qs)'}\|_{A^{(\frac qs)'}_\omega}\|G_{\mu,s}^v (f_{a,p})\|^s_{L^{q}_\omega}
    \lesssim \|G_{\mu,s}^v (f_{a,p})\|^s_{L^{q}_\omega}.
    \end{split}
    \end{equation*}
This together with \eqref{eq:c1} gives
    $
    \lim_{|a|\to 1^-}\frac{\muov(S(a))}{\omega(S(a))^{1+\frac{s}{p}-\frac{s}{q}}}=0
    $
and hence $I_d:A^p_\omega\to L^{p+s-\frac{ps}{q}}_{\muov}$ is compact by Theorem~\ref{Theorem:normscm}.

Conversely, let $q>s$ and assume that $I_d:\,A^p_\omega\to L^{p+s-\frac{ps}{q}}_{\muov}$  is compact.
Write $t=t(p,q,s)=p+s-\frac{ps}{q}>s$ for short.
Let $\{f_n\}$ be a bounded sequence in $A^p_\om$. Then, we may extract a subsequence
$\{f_{n_k}\}$ that converges in $L^t_{\muov}$ and uniformly on compact subsets to some $f\in A^p_\om$.
Write $g_{n_k}=f_{n_k}-f$. By Fubini's theorem, \cite[Theorem~3]{JAJRTent} and H\"older's inequality,
    \begin{equation*}
    \begin{split}
    \|G_{\mu,s}^v(g_{n_k})\|^s_{L^{q}_\omega}&=\sup_{\|h\|_{L^{(\frac{q}{s})'}_\omega}\le1}\int_\D|h(z)|
    \left(\int_{\Gamma(z)}|g_{n_k}(\z)|^s\frac{d\muov (\z)}{\omega(T(\z))}\right)\omega(z)\,dA(z)
    \lesssim \|g_{n_k}\|^s_{L^{t}_{\muov}}
    \end{split}
    \end{equation*}
and hence $$
\lim_{k\to\infty}\|G_{\mu,s}^v(g_{n_k})\|^s_{L^{q}_\omega}=0.
$$
If $s\ge 1$, two applications of Minkowski's inequality gives
    $$
    \left|\|G_{\mu,s}^v(f_{n_k})\|_{L^{q}_\omega}-\|G_{\mu,s}^v(f)\|_{L^{q}_\omega} \right|\le \|G_{\mu,s}^v(g_{n_k})\|_{L^{q}_\omega}\to 0,\quad k\to \infty.
    $$
Moreover, since $\{f_{n_k}\}$  converges uniformly on compact subsets of $\D$ to $f$, then $\vp_k(z)=\left(\int_{\Gamma(z)}|f_{n_k}(\z)|^s\frac{d\muov (\z)}{\omega(T(\z))}\right)^{1/s}$ converges to $\vp(z)=\left(\int_{\Gamma(z)}|f(\z)|^s\frac{d\muov (\z)}{\omega(T(\z))}\right)^{1/s}$ for
each $z\in\D$. Therefore Lemma~\ref{technichal} yields
    $$
    \lim_{k\to\infty}\|G_{\mu,s}^v(f_{n_k})-G_{\mu,s}^v(f)\|_{L^{q}_\omega}
    =\lim_{k\to\infty}\|\vp_k-\vp\|_{L^{q}_\omega}=0,
    $$
and thus $G_{\mu,s}^v:A^p_\omega\to L^q_\omega$ is compact.

If $0<s\le 1$,
    $$
    \int_{\Gamma(z)}|f_{n_k}(\z)|^s\frac{d\muov (\z)}{\omega(T(\z))}\le \int_{\Gamma(z)}|g_{n_k}(\z)|^s\frac{d\muov (\z)}{\omega(T(\z))}+\int_{\Gamma(z)}|f(\z)|^s\frac{d\muov (\z)}{\omega(T(\z))},
    $$
which together with Minkowski's inequality yields
    $$
    \left|\|G_{\mu,s}^v(f_{n_k})\|^s_{L^{q}_\omega}-\|G_{\mu,s}^v(f)\|^s_{L^{q}_\omega} \right|\le \|G_{\mu,s}^v(g_{n_k})\|^s_{L^{q}_\omega}.
    $$
Now, by arguing as in the previous case we see that $G_{\mu,s}^v:A^p_\omega\to L^q_\omega$ is compact. In view of~\eqref{Eq:case-q=s}, a similar reasoning also applies in the case $q=s$.

It remains to consider the case $q<s$. Assume first that $I_d:A^p_\omega\to L^{p+s-\frac{ps}{q}}_{\muov}$ is compact.
Let $\{f_n\}$ be a bounded sequence in $A^p_\om$. Then we may extract a subsequence
$\{f_{n_k}\}$ that converges on $L^{p+s-\frac{ps}{q}}_{\muov}$ and uniformly on compact subsets of $\D$ to
some $f\in A^p_\om$. Write $g_{n_k}=f_{n_k}-f$, and let $0<x<s$. H\"older's inequality and Fubini's theorem yield
    \begin{equation*}
    \begin{split}
    \|G_{\mu,s}^v(g_{n_k})\|^q_{L^{q}_\omega}
    &\le\int_\D N(g_{n_k})(z)^{\frac{qx}{s}}\left(\int_{\Gamma(z)}|g_{n_k}(\z)|^{s-x}\frac{d\muov(\z)}{\omega(T(\z))}\right)^\frac{q}{s}\omega(z)\,dA(z)\\
    & \lesssim\|N(g_{n_k})\|^{\frac{qx}{s}}_{L_\omega^{\frac{qx}{s-q}}}\|g_{n_k}\|^{\frac{q(s-x)}{s}}_{L_{\muov}^{(s-x)}}.
    \end{split}
    \end{equation*}
Take $x=\frac{p(s-q)}{q}<s$ so that $s-x=s+p-\frac{ps}{q}$. Then the estimates above together with \cite[Lemma 4.4]{PelRat}  give
    \begin{equation*}
    \begin{split}
    \|G_{\mu,s}^v(g_{n_k})\|^q_{L^{q}_\omega}
    \lesssim\|g_{n_k}\|^{\frac{qx}{s}}_{A_\omega^{p}}\|g_{n_k}\|^{\frac{q(s-x)}{s}}_{L_{\muov}^{s+p-\frac{ps}{q}}}
    \lesssim \|g_{n_k}\|^{\frac{q(s-x)}{s}}_{L_{\muov}^{s+p-\frac{ps}{q}}},
    \end{split}
    \end{equation*}
and hence
    $$
    \lim_{k\to\infty}\|G_{\mu,s}^v(g_{n_k})\|_{L^{q}_\omega}=0.
    $$
Now, by using Lemma~\ref{technichal} and arguing as in the case $q>s$, we conclude that
$\lim_{k\to\infty}\|G_{\mu,s}^v(f_{n_k})-G_{\mu,s}^v(f)\|_{L^{q}_\omega}
= 0$, that is, $G_{\mu,s}^v:A^p_\omega\to L^q_\omega$ is compact.

Conversely, let $q<s$ and assume that $G_{\mu,s}^v:A^p_\omega\to L^q_\omega$ is bounded. Choose $\alpha>\beta>1$ such that $\frac{\beta}{\alpha}=\frac{q}{s}$. By arguing as in \eqref{1-s} and \eqref{1-s'}, we get
    \begin{equation*}
    \begin{split}
    \frac{\muov(S(a))}{\om\left(S(a)\right)^{s/p}}
    &\lesssim\|G_{\mu,s}^v(f_{a,p})\|^{\frac{q}{\beta}}_{L^{q}_\omega}
    \|G_{\mu,1}^v(f^s_{a,p}(z)\chi_{S(a)})\|^{\frac{1}{\alpha'}}_{L^{\frac{\beta'}{\alpha'}}_\omega}
    \end{split}
    \end{equation*}
and
    \begin{equation*}
    \begin{split}
    \|G_{\mu,1}^v(f^s_{a,p}(z)\chi_{S(a)})\|_{L^{\frac{\beta'}{\alpha'}}_\omega}
    &\lesssim\frac{\muov(S(a))^{\frac{\alpha'}{\beta'}}}{\om\left(S(a)\right)^{s/p}} \left(\sup_{z\in \D} \frac{\muov(S(z)\cap S(a))}{\omega(S(z))} \right)^{1-\frac{\alpha'}{\beta'}},\quad a\in\D,
    \end{split}
    \end{equation*}
respectively. These estimates yield
    \begin{equation*}
    \begin{split}
    \frac{\muov(S(a))}{\om\left(S(a)\right)^{s/p}}
    \lesssim \|G_{\mu,s}^v(f_{a,p})\|^{\frac{q}{\beta}}_{L^{q}_\omega}\frac{\muov(S(a))^{\frac{1}{\beta'}}}{\om\left(S(a)\right)^{\frac{s}{p\alpha'}}} \left(\sup_{z\in \D} \frac{\muov(S(z)\cap S(a))}{\omega(S(z))} \right)^{\frac{1}{\alpha'}-\frac{1}{\beta'}},\quad a\in\D,
    \end{split}
    \end{equation*}
and thus
    \begin{equation}
    \begin{split}\label{eq:1-sc}
    \frac{\muov (S(a))}{\omega(S(a))^{\frac{q}{p}}}
    &\lesssim\|G_{\mu,s}^v(f_{a,p})\|^{q}_{L^{q}_\omega}
    \left(\sup_{z\in \D} \frac{\muov(S(z)\cap S(a))}{\omega(S(z))} \right)^{\frac{\beta}{\alpha'}-\frac{\beta}{\beta'}},
    \quad a\in\D.
    \end{split}
    \end{equation}
If $q=p$, we may use $\lim_{|a|\to 1^-}\|G_{\mu,s}^v (f_{a,p})\|_{L^{p}_\omega}=0$ and Theorem~\ref{th:casepleq} to deduce that the right-hand side of \eqref{eq:1-sc} tends to zero as $a$ approaches the boundary. Therefore $I_d:A^p_\om\to L^p_{\muov}$ is compact by Theorem~\ref{Theorem:normscm}(ii).

If $q>p$, by using \eqref{eq:1-sc} and arguing as in the corresponding part of the proof of Theorem~\ref{th:casepleq},
we get
    \begin{equation*}
    \begin{split}
    \frac{\muov(S(a))}{\omega(S(a))^{1+\frac{s}{p}-\frac{s}{q}}}
    &\le\|G_{\mu,s}^v(f_{a,p})\|^{q}_{L^{q}_\omega}
    \left(\sup_{b\in\D}\frac{\muov(S(b))}{\omega(S(b))^{1+\frac{s}{p}-\frac{s}{q}}}\right)^{1-\frac{q}{s}},
    \quad a\in\D.
    \end{split}
    \end{equation*}
from which arguments similar to those applied in the previous paragraph show that $I_d:A^p_\om\to L^{p+s-\frac{ps}{q}}_{\muov}$ is compact.
This finishes the proof.
\end{proof}

\begin{theorem}
Let $0<q<p<\infty$ and $0<s<\infty$ such that $1+\frac{s}p-\frac{s}{q}>0$, $\omega\in\DD$ and let $\mu,v$ be  positive Borel measures on $\D$
such that $\mu\left(\{z\in\D: v(T(z))=0\}\right)=0=\mu(\{0\})$. Then the following conditions are equivalent:
    \begin{enumerate}
    \item[\rm(i)] $G_{\mu,s}^v:A^p_\omega\to L^q_\omega$ is compact;
    \item[\rm(ii)] $G_{\mu,s}^v:A^p_\omega\to L^q_\omega$ is bounded;
    \item[\rm(iii)] $I_d: A^p_\om\to L^{p\left(1+\frac{s}p-\frac{s}{q}\right)}_{\muov}$ is compact;
    \item[\rm(iv)] $I_d: A^p_\om\to L^{p\left(1+\frac{s}p-\frac{s}{q}\right)}_{\muov}$ is bounded.
    \end{enumerate}
\end{theorem}

\begin{proof}
The conditions (ii)--(iv) are equivalent by Theorems~\ref{Theorem:main-s}(i) and~\ref{Theorem:normscm}(iii). To complete the proof, it suffices to show that $G_{\mu,s}^v:A^p_\omega\to L^q_\omega$ is compact if $I_d: A^p_\om\to L^{p\left(1+\frac{s}p-\frac{s}{q}\right)}_{\muov}$ is bounded. To see this, let $\{f_n\}$ be a bounded sequence in $A^p_\om$, let $\{f_{n_k}\}$ be a subsequence that converges uniformly on compact subsets of $\D$ to $f\in A^p_\om$. Write $g_k=f_{n_k}-f$ as before. By using Theorem~\ref{Theorem:normscm} and the last part of the proof of Theorem~\ref{Theorem:normscm}(iii), we deduce
    $$
    \lim_{R\to 1^-} \int_\D \left(\int_{\Gamma(\z)\setminus\overline{D(0,R)}} \frac{d\mu(z)}{\om(T(z))}\right)^{\frac{pq}{s(p-q)}}\,\om(\z)\,dA(\z)=0.
    $$
Therefore, for a fixed $\e>0$, there exists $R_0\in(0,1)$ such that
    $$
    \int_\D \left(\int_{\Gamma(\z)\setminus\overline{D(0,R_0)}} \frac{d\mu(z)}{\om(T(z))}\right)^{\frac{pq}{s(p-q)}}\,\om(\z)\,dA(\z)<\e^\frac{p}{p-q}.
    $$
Choose $k_0\in\N$ such that $|g_k(z)|<\e^{1/q}$ for all $k\ge k_0$ and $z\in\overline{D(0,R_0)}$. Then H\"older's inequality and \cite[Lemma~4.4]{PelRat} give
    \begin{equation*}
    \begin{split}
    \|G_{\mu,s}^v(g_k)\|^q_{L^{q}_\omega}
    &\lesssim  \int_\D\left(\int_{\Gamma(z)\cap\overline{D(0,R_0)}}|g_k(\z)|^s\frac{d\mu(\z)}{v(T(\z))}\right)^\frac{q}{s}\omega(z)\,dA(z)\\
    &\quad+\int_\D\left(\int_{\Gamma(z)\setminus\overline{D(0,R_0)}}|g_k(\z)|^s\frac{d\mu(\z)}{v(T(\z))}\right)^\frac{q}{s}\omega(z)\,dA(z)\\
    &\lesssim \e+
    \int_\D N(g_k)^q(z)\left(\int_{\Gamma(z)\setminus\overline{D(0,R_0)}}\frac{d\mu(\z)}{v(T(\z))}\right)^\frac{q}{s}\omega(z)\,dA(z)\\
    &\le \e+
    \|N(g_k) \|^q_{L^{p}_\omega}\left(\int_\D \left(\int_{\Gamma(z)\setminus\overline{D(0,R_0)}}\frac{d\mu(\z)}{v(T(\z))}\right)^\frac{pq}{(p-q)s}\omega(z)\,dA(z)\right)^\frac{p-q}{p}
    \lesssim \e,
    \end{split}
    \end{equation*}
and consequently, $\lim_{k\to \infty}\|G_{\mu,s}^v(g_k)\|^q_{L^{q}_\omega}=0$. Finally, by using Lemma~\ref{technichal} and arguing as in the proof of Theorem~\ref{th:c1}, we deduce $\lim_{k\to \infty}\|G_{\mu,s}^v(f_k)-G_{\mu,s}^v(f)\|^q_{L^{q}_\omega}=0$, that is, $G_{\mu,s}^v:A^p_\omega\to L^q_\omega$ is compact.
\end{proof}

\section{Applications and further comments}\label{Sec:comments}

\subsection{Area operators in Hardy spaces}\label{Subsection:Area-Hardy}

For $0<s<\infty$, define
    $$
    G_{\mu,s}(f)(z)=\left(\int_{\Gamma(z)}|f(\z)|^s\frac{d\mu (\z)}{1-|\z|}\right)^\frac1s,\quad z\in\T.
    $$
The method of proof of Theorem~\ref{Theorem:main-s}, combined with the results in \cite[Section $7$]{Lu90}
and \cite{PelPAMS14}, can be used to obtain the following result. The details of the proof does not reveal anything new, and are therefore omitted. Here $L^q(\T)$ refers to the classical $L^q$-space on $\T$.

\begin{theorem}\label{Theorem:main-sHp}
Let $0<p,q<\infty$ such that $1+\frac{s}p-\frac{s}{q}>0$, and let $\mu$ be a positive Borel measure on $\D$ such that $\mu(\{0\})=0$. Then $G_{\mu,s}:H^p\to L^q(\T)$ is bounded
 (resp. compact) if and only if $I_d:\,H^p\to L^{p+s-\frac{ps}{q}}_\mu$ is bounded (resp. compact).
Moreover,
    \begin{equation*}
    \begin{split}
    \|G_{\mu,s}\|^s_{H^p\to L^{q}(\T)}
    \asymp\|Id\|^{p+s-\frac{ps}{q}}_{H^p\to L^{p+s-\frac{ps}{q}}_{\mu}}
    \asymp\sup_{a\in\D}\frac{\mu(S(a))}{(1-|a|)^{1+\frac{s}{p}-\frac{s}{q}}},\quad p\le q,
    \end{split}
    \end{equation*}
and
    \begin{equation*}
    \begin{split}
    \|G_{\mu,s}\|^s_{H^p\to L^{q}(\T)}
    \asymp\|Id\|^{p+s-\frac{ps}{q}}_{H^p\to L^{p+s-\frac{ps}{q}}_{\mu}}
    \asymp\|B_\mu\|_{L^{\frac{qp}{s(p-q)}}},\quad q<p,
    \end{split}
    \end{equation*}
where
    $$
    B_\mu(\z)=\int_{\Gamma(\z)}\frac{d\mu(z)}{1-|z|},\quad
    \z\in\T.
    $$
\end{theorem}

In particular, this result proves the conjecture in \cite[p.~365]{GonLouWu2010} in the case $1+\frac{1}p-\frac{1}{q}>0$.

\subsection{Integral operator $T_g$ on Bergman and Hardy spaces}

Each $g\in\H(\D)$ induces the integral operator
    $$
    T_g(f)(z)=\int_0^zg'(\z)f(\z)\,d\z,\quad z\in\D,
    $$
acting on $\H(\D)$. This type of integral operators have been extensively studied
during the last decades and have interesting connections with other areas of mathematical analysis, see \cite{PelRat,PelSum14} and the references therein. In particular, the symbols $g$ for which $T_g$ is bounded or compact from $A^p_\om$ to
$A^q_\om$ can be described in terms of the following spaces of analytic functions when $q\ge p$.

We say that $g\in\H(\D)$ belongs to $\CC^{q,p}(\om^\star)$,
$0<p,q<\infty$, if the measure $|g'(z)|^2\om^\star(z)\,dA(z)$ is a
$q$-Carleson measure for $A^p_\om$. Moreover,
$g\in\CC^{q,p}_0(\om^\star)$ if the identity operator
$I_d:A^p_\omega\to L^q(|g'|^2\omega^\star dA)$ is compact. If
$q\ge p$ and $\om\in\DD$, then Theorem~\ref{Theorem:normscm} shows that
these spaces only depend on the quotient $\frac{q}{p}$.
Consequently, for $q\ge p$ and $\om\in\DD$, we simply write
$\CC^{q/p}(\om^\star)$ instead of $\CC^{q,p}(\om^\star)$. Thus,
if $\alpha\ge 1$ and $\om\in\DD$, then
$\CC^{\alpha}(\om^\star)$ consists of those $g\in\H(\D)$ such
that
    \begin{equation}\label{calpha}
    \|g\|^2_{\CC^{\alpha}(\om^\star)}=|g(0)|^2+\sup_{I\subset\T}\frac{\int_{S(I)}|g'(z)|^2\om^\star(z)\,dA(z)}
    {\left(\om\left(S(I)\right)\right)^{\alpha}}<\infty.\index{$\CC^\a(\om^\star)$}
    \end{equation}
An analogue of this identity is valid for the little space
$\CC^{\alpha}_0(\om^\star)$. We refer to \cite[Chapter~5]{PelRat}
for further information about these spaces.

\begin{theorem}\label{Theorem:T_g-Bergman}
Let $0<p,q<\infty$ such that $q>\frac{2p}{2+p}$ and $\om\in\DD$. Let $g\in \H(\D)$ and denote $d\mu_g(z)=|g'(z)|^2\om^\star(z)\,dA(z)$. Then $T_g:A^p_\om\to A^q_\om$ is bounded (resp. compact) if and only if $I_d:A^p_\omega\to L^{p+2-\frac{2p}{q}}_{\mu_g}$ is bounded (resp. compact).
\end{theorem}

\begin{proof}
By \cite[Theorem~4.2]{PelRat}, $T_g:A^p_\om\to L^q_\om$ is bounded  (resp. compact) if and only if $G_{\mu_g,2}^\om:A^p_\om\to L^q_\om$ is bounded (resp. compact), and
    $
    \|T_g\|_{A^p_\om\to L^q_\om}\asymp\|G_{\mu_g,2}^\om\|_{A^p_\om\to L^q_\om}
    $.
Theorem~\ref{Theorem:main-s}
implies
    $$
    \|G_{\mu_g,2}^\om\|^2_{A^p_\om\to L^q_\om}
    \asymp\|I_d\|^{p+2-\frac{2p}{q}}_{A^p_\omega\to L^{p+2-\frac{2p}{q}}_{\mu_g}},
    $$
and this finishes the proof.
\end{proof}

It is worth mentioning that Theorem~\ref{Theorem:normscm} yields
    $$
    \|T_g\|_{A^p_\om\to L^q_\om}^2
    \asymp\sup_{a\in\D}\frac{\mu_g(S(a))}{\om(S(a))^{2(\frac1p-\frac1q)+1}},\quad q\ge p
    $$
Thus $T_g:A^p_\om\to L^q_\om$ is bounded if and only if $g\in\mathcal{C}^{2(\frac1p-\frac1q)+1}(\om^\star)$. If $p>q$, Theorem~\ref{Theorem:normscm} also gives
    $$
    \|T_g\|_{A^p_\om\to L^q_\om}^2
    \asymp\int_\D\left(\int_{\Gamma(\z)}|g'(z)|^2\,dA(z)\right)^\frac{qp}{2(p-q)}\om(\z)\,dA(\z),
    $$
and thus $T_g:A^p_\om\to L^q_\om$ is bounded if and only if $g\in A^\frac{qp}{p-q}_\om$ by \cite[Theorem~4.2]{PelRat}.

Consequently, whenever $q>\frac{2p}{2+p}$
Theorem~\ref{Theorem:T_g-Bergman} improves \cite[Theorem~4.1]{PelRat} because the hypothesis on $\om$ are stronger in the original result.
In particular, if $q<p$, the weight $\om$ is assumed to be continuous and strictly positive with the local regularity
    $$
    \om(t)\asymp\om(r),\quad 1-t\asymp1-r.
    $$
This hypothesis allows one to use the strong factorization $A^p_\om=A^{p_1}_\om\cdot A^{p_2}_\om$, $p^{-1}=p_1^{-1}+p_2^{-1}$,
\cite[Theorem~3.1]{PelRat}, which is a principal ingredient in the proof of \cite[Theorem~4.1]{PelRat}.

However, the defect of Theorem~\ref{Theorem:T_g-Bergman} is the extra hypothesis $q>\frac{2p}{2+p}$ which is a restriction only in the case $p>q$. This condition is inherited from Theorem~\ref{Theorem:main-s} and appears there because Carleson measures are finite measures. This is not true in general for~$\mu_g$ when $g\in A^\frac{qp}{p-q}_\om$ and $q<\frac{2p}{2+p}$.
The case of compact operators can be analyzed in the same way.

If $q>\frac{2p}{2+p}$  one may characterize bounded and compact operators $T_g:H^p\to H^q$ by using Section~\ref{Subsection:Area-Hardy}. In order to avoid unnecessary repetition, we omit the details.

\end{document}